\documentclass[12pt,fleqn,draft]{article}
\usepackage{amsmath,amssymb,amsthm,amsfonts,bm}
\usepackage{enumerate}
\usepackage{indentfirst}
\usepackage{cite}
\usepackage[usenames,dvipsnames]{color}
\definecolor{blu}{rgb}{0,0,0.7}
\definecolor{rosso}{rgb}{0.85,0,0}
\definecolor{purple}{rgb}{0.63, 0.36, 0.94}
\definecolor{verde}{rgb}{0, 0.5, 0.2}
\definecolor{palecopper}{rgb}{0.85, 0.54, 0.4}

\def\pier #1{{\color{black}#1}}
\def\sk #1{{\color{black}#1}}
\def\pcol #1{{\color{black}#1}}
\def\rev #1{{\color{black}#1}}

\makeatletter
\def\@cite#1#2{[{{\bfseries #1}\if@tempswa , #2\fi}]}
\renewcommand{\section}{%
\@startsection{section}{1}{\z@}
{0.5truecm plus -1ex minus -.2ex}%
{1.0ex plus .2ex}{\bfseries\large}}
\def\@seccntformat#1{\csname the#1\endcsname.\ }
\makeatother

\numberwithin{equation}{section} 
\newtheorem{thm}{Theorem}[section]

\newtheorem{lem}[thm]{Lemma}

\theoremstyle{definition}
\newtheorem{df}{Definition}[section]
\newtheorem{remark}{Remark}[section]
\newtheorem*{prth2.1}{Proof of Theorem 2.1}
\newtheorem*{prth2.2}{Proof of Theorem 2.2}
\newtheorem*{prth3.1}{Proof of Theorem 3.1}

\newcommand{\ep}{\varepsilon}

\def\iot {\int_0^t}

\def\iO{\int_\Omega}

\def\checkmmode #1{\relax\ifmmode\hbox{#1}\else{#1}\fi}

\let\hat\widehat

\def\Pi{\hat\pi}

\def\Vp{V^*}

\let\non\nonumber

\begin{document}
\footnote[0]
    {2010 {\it Mathematics Subject Classification}\/: 
    \pier{35G31; 35D30; 35A01; 80A22.}}
\footnote[0]
    {{\it Key words and phrases}\/: 
    \pier{conserved phase field system; inertial term;  non-isothermal process;  
    nonlinear partial differential equations; initial-boundary value problem; existence of solutions.}} 
%==========================title==========================
\begin{center}
    \Large{{\bf \rev{Analysis of a nonisothermal and conserved \\ phase field system 
    with inertial term}}}
\end{center}
\vspace{5pt}
%===========================author=========================
\begin{center}
    Pierluigi Colli\\
    \vspace{2pt}
    Dipartimento di Matematica ``F. Casorati", Universit\`a di Pavia \\ 
    and Research Associate at the IMATI -- C.N.R. Pavia \\ 
    via Ferrata 5, 27100 Pavia, Italy\\
    {\tt pierluigi.colli@unipv.it}\\
    \vspace{12pt}
    Shunsuke Kurima%
   \footnote{Corresponding author}\\
    \vspace{2pt}
    Department of Mathematics, 
    Tokyo University of Science\\
    1-3, Kagurazaka, Shinjuku-ku, Tokyo 162-8601, Japan\\
    {\tt shunsuke.kurima@gmail.com}\\
    \vspace{2pt}
\end{center}
%\begin{center}    
%    \small \today
%\end{center}

\vspace{2pt}
%=====================  Abstract  =======================
\newenvironment{summary}
{\vspace{.5\baselineskip}\begin{list}{}{%
     \setlength{\baselineskip}{0.85\baselineskip}
     \setlength{\topsep}{0pt}
     \setlength{\leftmargin}{12mm}
     \setlength{\rightmargin}{12mm}
     \setlength{\listparindent}{0mm}
     \setlength{\itemindent}{\listparindent}
     \setlength{\parsep}{0pt}
     \item\relax}}{\end{list}\vspace{.5\baselineskip}}
\begin{summary}
{\footnotesize {\bf Abstract.}
     This paper deals with a conserved phase field system 
\pier{that couples the energy balance equation with a Cahn--Hilliard type system
including temperature and the inertial term for the order parameter.}
%\begin{equation*}
%\begin{cases}
%\theta_t + \varphi_t - \Delta_{N}\theta = f  & \mbox{in}\ \Omega\times(0, T), \\[1mm]
%\tau\varphi_{tt} + \varphi_{t} - \Delta_{N}\mu = 0 
%& \mbox{in}\ \Omega\times(0, T), \\[1mm]
%- \Delta_{N}\varphi + \beta(\varphi) + \pi(\varphi) = \mu + \theta 
%& \mbox{in}\ \Omega\times(0, T).
%\end{cases}
%\end{equation*}
%Here $\Omega \subset \mathbb{R}^d$ ($d = 1, 2, 3$) is a bounded domain, 
%$T>0$, $\tau > 0$, $\Delta_{N}$ is the homogeneous Neumann Laplacian, 
%$\beta$ is a monotone function, $\pi$ is an anti-monotone function, 
%$f$ is a give function. 
%We note that the nonlinear term $\beta + \pi$ 
%represents the derivative of a potential of double well type.  
In the case \pier{without inertial term, 
the system under study was introduced by Caginalp. 
The inertial term is motivated by the occurrence of}
rapid phase transformation processes in nonequilibrium dynamics. \pcol{A double-well potential is well chosen and the related nonlinearity governing the evolution is assumed to satisfy a suitable growth condition.
\pcol{The viscous variant of the Cahn--Hilliard system is also considered along with the inertial term.}}
The existence of a global solution is proved via \pcol{the analysis of some approximate problems with Yosida regularizations, and the use of the Cauchy--Lipschitz--Picard theorem in an abstract setting. Moreover, we study 
the convergence of the system, with or without the viscous term, as the inertial coefficient tends to zero.}}

\end{summary}
\vspace{10pt}

\newpage

%\setlength{\baselineskip}{1.1\baselineskip}

%%==============================================================%%
%%==============                                  ==============%%
%%======                      Section1                    ======%%
%%====                                                      ====%%
%%==                                                         ==%%
%%====               Introduction          ====%%
%%======                                                  ======%%
%%==============                                  ==============%%
%%==============================================================%%

\section{Introduction} \label{Sec1}

{In this paper, we address \pier{a nonisothermal and conserved phase field system \rev{that} includes an inertial term 
for the order parameter. We \rev{analyze} the related initial-boundary value problem from \rev{a mathematical perspective 
and establish the} existence of solutions.}

\smallskip
\pier{Let us start by recalling that in the \rev{1990s} Caginalp introduced 
the following phase field system~\cite{Caginalp1, Caginalp2}}
 \begin{align} 
&\theta_t + \varphi_t - \Delta\theta = f, \label{e1}
\\[1mm]
&\varphi_t - \Delta\mu = 0, \label{e2}
\\[1mm]
&-\Delta\varphi + g(\varphi) = \mu + \theta, \label{e3}
\end{align}
\rev{where $\theta$ denotes the relative temperature (around a critical value normalized to zero)}, 
$\varphi$ is the order parameter \rev{and $f$ is a given a source term. The 
auxiliary variable $\mu$ serves to combine the second and third equations into a single fourth-order formulation.}
Moreover, $g$ is the derivative of a double well potential $G$\pier{\rev{:~a} typical choice for the double well potential is $G(r) = \dfrac{1}{4}(r^2 -1)^2$, 
so that  the nonlinearity in \eqref{e3} reduces to the usual cubic nonlinear term $g(r) = r^3 - r$, with $r $ denoting the real 
variable. \rev{In our work, we adopt the} decomposition $g = \beta + \pi$,  
with $\beta$ monotone increasing and with an anti-monotone decreasing function $\pi$.}

\smallskip
\pier{The system \eqref{e2}-\eqref{e3}, or the resulting combined equation, also coupled to other sets of equations for suitable physical models, is universally referred to Cahn and Hilliard~\cite{CH} and is the subject of a great number of recent investigations in the mathematical literature (see \cite{CM, CoGaMi, SM, CHbook} among many other contributions).}

\smallskip
\pier{About \eqref{e1}-\eqref{e3}, we can underline that this system is important for the modelling of 
phase segregation processes under non-isothermal conditions, in particular for  binary mixtures undergoing a phase 
separation at non-constant temperature. Different mathematical studies have been devoted to this class of problems (see, e.g., \cite{Bro, BHN, G2007, M2013, CGSS1, CGSS2}; in particular, in \cite{CGSS1} you can find a recent overview).}} 

\smallskip
\pier{This paper deals with a variation of the system \eqref{e1}-\eqref{e3} in which we include the inertial term for the 
order parameter. Indeed, in order to account for rapid phase transformation processes in nonequilibrium dynamics, the term $\pier{\tau} \varphi_{tt}$ is added in equation~\eqref{e2}. As far as we know, similar investigations have been carried out for the Cahn--Hilliard system \rev{or more general systems including the Cahn--Hilliard structure} by considering hyperbolic relaxations of it (cf., e.g., \rev{\cite{Bon1, Bon2, CGW, Dor, DMP, GPS}}).}

\pier{\rev{The specific system we analyze is given by:}
\begin{equation}\label{Ptau}\tag*{(P)$_{\tau}$}
\begin{cases}
\theta_t + \varphi_t - \Delta\theta = f  & \mbox{in}\ \Omega\times(0, T), \\[1mm]
\tau \varphi_{tt} + \varphi_{t} - \Delta\mu = 0 
& \mbox{in}\ \Omega\times(0, T), \\[1mm]
- \Delta\varphi + \beta(\varphi) + \pi(\varphi) = \mu + \theta 
& \mbox{in}\ \Omega\times(0, T), \\[1mm]
\partial_{\nu} \theta =  \partial_{\nu} \mu = \partial_{\nu} \varphi = 0
& \mbox{on}\ \partial\Omega\times(0, T), \\[1mm]
\theta(0) = \theta_0,\ \varphi(0) = \varphi_0,\ \varphi_{t}(0) = v_0 
& \mbox{in}\ \Omega. 
\end{cases}
\end{equation}
Here, $\Omega \subset \mathbb{R}^d$ ($d = 1, 2, 3$) is a bounded and connected domain with smooth boundary $\partial \Omega$
and $T>0$ denotes some final time. Moreover, $\tau > 0$ represents the inertial coefficient, $\Delta$ is the Laplacian and
$\partial_{\nu}$ denotes the outward normal derivative on the boundary. As for 
$\beta$, $\pi$, the known function $f$ and the initial data we set suitable assumptions.}
\sk{\pcol{In addition to the problem~\ref{Ptau}, we also consider a viscous variant of it which has an interest in itself and consists in including a term proportional to the time derivative of $\varphi$ in the third equation of~\ref{Ptau}. Indeed, we introduce the problem 
\begin{equation}\label{Ptaueta}\tag*{(P)$_{\tau, \eta}$}
\begin{cases}
\theta_t + \varphi_t - \Delta\theta = f  & \mbox{in}\ \Omega\times(0, T), \\[1mm]
\tau \varphi_{tt} + \varphi_{t} - \Delta\mu = 0 
& \mbox{in}\ \Omega\times(0, T), \\[1mm]
\eta\varphi_{t} - \Delta\varphi + \beta(\varphi) + \pi(\varphi) = \mu + \theta 
& \mbox{in}\ \Omega\times(0, T), \\[1mm]
\partial_{\nu} \theta =  \partial_{\nu} \mu = \partial_{\nu} \varphi = 0
& \mbox{on}\ \partial\Omega\times(0, T), \\[1mm]
\theta(0) = \theta_0,\ \varphi(0) = \varphi_0,\ \varphi_{t}(0) = v_0 
& \mbox{in}\ \Omega, 
\end{cases}
\end{equation}
for $\eta \in (0, 1]$, and we also deal with the analysis of~\ref{Ptaueta}.}}

\medskip

\pcol{To establish a global existence result for the problem~{\ref{Ptaueta}} we first approximate the monotone nonlinearity
$\beta$ using its Yosida regularization, and this will be the approximation with parameter $\ep>0$. At the second level of approximation, with parameter $\lambda >0$, we replace the operator $-\Delta$ in $L^2(\Omega) $ with its Yosida regularization. The resulting problem is then analyzed using the Cauchy--Lipschitz--Picard theorem. Subsequently, we derive uniform estimates for both approximations and employ compactness arguments to pass to the limit, thereby proving the existence of solutions not only for {\ref{Ptaueta}} but also for \ref{Ptau} as $\eta \searrow 0$. Furthermore, we investigate the asymptotic behavior of both {\ref{Ptaueta}} and \ref{Ptau} as $\tau $ approaches zero.} 

\rev{While certain aspects of our analysis can be compared to those in~\cite{CGW , GPS}, our approach introduces some novel aspects. Firstly, we consider a nonisothermal conserved phase field system that includes a forcing term $f$ in the phase equation -- an element not present in previous studies. Secondly, our treatment of the nonlinear potential is structurally distinct: rather than using a single nonlinear function, we adopt a decomposition into $ \beta + \pi$, which accommodates a broader class of nonlinearities. From a methodological standpoint, our analysis diverges from the traditional use of Faedo--Galerkin approximations. Instead, we employ Yosida regularizations for second-order operators and leverage the abstract framework of nonlinear evolution equations. This alternative strategy provides a different and versatile analytical pathway.}

\rev{Finally, we prove the convergence of solutions as the inertial coefficient approaches zero, thereby rigorously linking the hyperbolic system to its parabolic counterpart. To the best of our knowledge, this singular perturbation analysis represents a new contribution in this setting.}

\rev{In conclusion, while our analytical framework builds on established abstract theory, this work addresses a nonisothermal conserved phase field model with inertial effects -- a setting that has received limited attention in the literature. The combined presence of temperature dependence and a second-order time derivative introduces significant mathematical challenges, particularly in proving global well-posedness and analyzing the vanishing inertial limit. Our approach, which includes deriving uniform estimates independent of the inertial parameter and handling coupled nonlinearities in the thermal and phase field equations, offers tools that may extend to other thermodynamically consistent systems involving memory, hysteresis, or fluid-structure interactions. We believe this study provides a solid foundation for further investigations into similar models in non-equilibrium thermodynamics, complex fluids, and biological systems.}

\section{Main results} \label{Sec1bis}

\pier{Regarding the data of the problem \ref{Ptau},  we assume} the following conditions (C1)-(C3): 
\begin{enumerate} 
\setlength{\itemsep}{1mm}
\item[(C1)] $\beta : \mathbb{R} \to \mathbb{R}$                                
is monotone and continuous \pier{with $\beta(0)=0$.}
We denote by $\widehat{\beta} : \mathbb{R} \to [0, +\infty)$ 
the unique $C^1$ convex function such that 
$\widehat{\beta}(0) = 0$ and $\widehat\beta'=\beta$: 
as usual,  we identify
$\beta$ with the subdifferential $\partial\widehat\beta$
in the sense of monotone analysis. We further assume that there are
an exponent $q>1$ and a constant $c_{\beta} > 0$ such that 
\[
  |\beta(r)|^q \leq c_{\beta}\bigl( 1 + \widehat\beta(r) \bigr)
\quad\forall\,r\in\mathbb R,
\]
where also $q\geq 6/5$ if $d=3$.
\item[(C2)] $\pi : \mathbb{R} \to \mathbb{R}$ 
is Lipschitz continuous. 
\item[(C3)] $f \in L^2(0, T; (H^1(\Omega))^{*})$, 
$\theta_0 \in L^2(\Omega)$, $\varphi_0 \in H^1(\Omega)$, 
$\widehat{\beta}(\varphi_0) \in L^{1}(\Omega)$, 
$v_0 \in (H^1(\Omega))^{*}$.
\end{enumerate}

\smallskip

\begin{remark}\label{rem1}
  Let us comment on the condition (C1). 
  First of all, we recall the embeddings
  $H^1(\Omega)\hookrightarrow L^\infty(\Omega)$ in $d=1$, 
  $H^1(\Omega)\hookrightarrow L^p(\Omega)$ for all $p\in[1,+\infty)$ in $d=2$, 
  and $H^1(\Omega)\hookrightarrow L^6(\Omega)$ in $d=3$. 
  Hence, it readily holds that
  \[
  L^q(\Omega)\subset (H^1(\Omega))^*
  \qquad\text{in every space dimension } d=1,2,3\pier{.}
  \]
  \pier{At the same time,} let us briefly discuss 
  the growth condition in (C1): note that $\beta$ can have any polynomial
  growth at infinity in dimensions $d=1,2$, 
  and any polynomial growth up to \pier{the}
  $5$-th order in dimension $d=3$.
  In particular, we note that the classical fourth order double-well potential, 
  leading to $\beta (r) = r^3$ and $\pi (r) = - r$, with $r\in \mathbb{R}$,  
  satisfies (C1) and (C2) in every space dimension.
\end{remark}

\smallskip
Let us \pier{also introduce}  the Hilbert spaces 
  \begin{align*} 
   H:=L^2(\Omega), \quad V:=H^1(\Omega)
\end{align*} 
 with inner products 
\begin{align*} 
&(u_{1}, u_{2})_{H}:=\int_{\Omega}u_{1}u_{2}\,dx \quad  (u_{1}, u_{2} \in H), 
\\
&(v_{1}, v_{2})_{V}:=
 \int_{\Omega}\nabla v_{1}\cdot\nabla v_{2}\,dx + \int_{\Omega} v_{1}v_{2}\,dx \quad 
 (v_{1}, v_{2} \in V),
\end{align*}
respectively,
and with the related Hilbertian norms.
Moreover, we \pier{set} 
\begin{align*} 
  W:=\bigl\{z\in H^2(\Omega)\ |\ \partial_{\nu}z = 0 \quad 
        \mbox{a.e.\ on}\ \partial\Omega\bigr\}.
\end{align*} 
The notation $V^{*}$ denotes the dual space of $V$ with 
duality pairing $\langle\cdot, \cdot\rangle_{V^*, V}$.
\pier{We identify $H$ with a subspace of $V$ in the usual way, i.e.,
in order that 
\begin{align*}
	\text{$\langle u, v \rangle_{V^*, V}=(u,v)_{H}$ 
	\quad 
	for every $u\in H$ and $v\in V$.}
\end{align*}
This makes $(V,H,V^*)$ a Hilbert triplet.
We notice that all of the embeddings
$
  W \hookrightarrow V \hookrightarrow H \hookrightarrow V^*
$ are dense and compact.
For any element $v\in V^*$ we can define the (generalized) mean value $m(v) $ by
\begin{equation}
\hbox{$m(v) := \dfrac{1}{|\Omega|}\langle v, 1 \rangle_{V^*, V}$ \quad $(v \in V^*)$,}
\label{pier8}
\end{equation}
where $1$ stands for the constant function 1 in $V$.
Moreover, we also introduce the duality mapping $F : V \to V^{*}$ defined by
    \begin{align}
    &\langle Fv_{1}, v_{2} \rangle_{V^*, V} := 
    (v_{1}, v_{2})_{V} \ \quad (v_{1}, v_{2}\in V),   
    \label{defF}
\end{align}
so that we also have
\begin{align}
    &(v_{1}^{*}, v_{2}^{*})_{V^{*}} := 
    \left\langle v_{1}^{*}, F^{-1}v_{2}^{*} 
    \right\rangle_{V^*, V} 
    \ \quad (v_{1}^{*}, v_{2}^{*}\in V^{*}).   
    \label{innerVstar}
    \end{align}}%   
%
%   SOME DEFINITIONS BY PIER
%
\def\<#1>{\mathopen\langle #1\mathclose\rangle_{V^*, V}}%
\def\norma #1{\mathopen \| #1\mathclose \|}%
\def\normaVp #1{\norma{#1}_*}%
\pier{We need another tool, represented by the operator $\cal N$.  Consider, for $\psi\in V^*$, the problem of finding
\begin{equation}
  u \in V
  \quad \hbox{such that} \quad
  \int_\Omega \nabla u \cdot \nabla v
  = \< \psi , v >
  \quad \hbox{for every $v\in V$}.
  \label{neumann}
\end{equation}
Now, since $\Omega$ is connected, it is not difficult to see that, for $\psi\in V^*$, \eqref{neumann} is solvable if and only if $\psi$ has zero mean value. Moreover, in this case {there is a unique solution} $u$ with zero mean value.
This {entails} that the operator \begin{align}
  & {\cal N}: D({\cal N}) := \{\psi\in\Vp:\ m(\psi) =0\} \to \{u\in V:\ m(u)=0\},
   \non
  \\
  & \hbox{for }\, \psi \in D({\cal N}), \, \ {\cal N} (\psi)\, \hbox{ is the unique $u$ solving \eqref{neumann} and  $m(u)=0$},
  \qquad\quad
  \label{defN}
\end{align}
is well defined and {yields} an isomorphism between the above spaces.
{In addition,} the formula
\begin{equation}
  \normaVp\psi^2 := \norma{\nabla{\cal N}(\psi-m(\psi))}^2 + |m(\psi)|^2
  \quad \hbox{for $\psi\in\Vp$}
  \label{normaVp}
\end{equation}
defines another norm in $\Vp$ that is equivalent to the standard norm $\|\,\cdot\,\|_{V^*}$.
From the above definitions one trivially derives~that
\begin{align}
  & \iO \nabla{\cal N}\psi \cdot \nabla v
  = \< \psi , v >
  \quad \hbox{for every $\psi\in D({\cal N})$ and $v\in V$},
  \label{dadefN}
  \\[2mm]
  & \< \psi , {\cal N}\zeta >
  = \< \zeta , {\cal N}\psi >
  \quad \hbox{for every $\psi,\zeta\in D({\cal N})$},
  \label{simmN}
  \\[2mm]  
  & \< \psi , {\cal N}\psi > 
  = \iO |\nabla{\cal N}\psi|^2
  = \normaVp\psi^2
  \quad \hbox{for every $\psi\in D({\cal N})$}.
  \label{danormaVp}
\end{align}
Moreover, it turns out that
\begin{align}
  &\iot \< \sk{v_{t}(s)} , {\cal N} v(s) > \, ds
  = \iot \< v(s) , \sk{{\cal N}(v_{t}(s))} > \, ds \notag\\
  &= \frac 12 \, \normaVp{v(t)}^2
  - \frac 12 \, \normaVp{v(0)}^2
  \label{propN} 
\end{align}
for every $t\in[0,T]$ and every $v\in H^1 (0,T;V^*)$ satisfying $m(v)=0$ a.e. in $(0,T)$.}
  
%We also consider \pier{the following subspace of $V$, i.e.}
%\begin{align*} 
%V_{0}:=\left\{ z \in H^1(\Omega)\ \Big{|} \ \int_{\Omega} z = 0 \right\}   
%\end{align*} 
%which \pier{is a Hilbert space with inner product
% \begin{align*} 
% (v_{1}, v_{2})_{V_{0}}:=
% \int_{\Omega}\nabla v_{1}\cdot\nabla v_{2}\,dx %+ \int_{\Omega} v_{1}v_{2}\,dx 
% \quad (v_{1}, v_{2} \in V_{0}) 
%\end{align*}
%and with the related norm.}
%The notation $V_{0}^{*}$ denotes the dual space of $V_{0}$ with 
%duality pairing $\langle\cdot, \cdot\rangle_{V_{0}^*, V_{0}}$. 
%Moreover, \pier{we can introduce} the bijective mapping 
%${\cal N} : V_{0}^{*} \to V_{0}$ and 
%the inner product in $V_{0}^{*}$ are specified by  
%    \begin{align}
%    &\langle v^{*}, v \rangle_{V_{0}^*, V_{0}}  =:  
%    \int_{\Omega} \nabla {\cal N}v^{*} \cdot \nabla v 
%    \quad (v^{*} \in V_{0}^*, v \in V_{0}),   
%    \label{defN}
%    \\[1mm]
%    &(v_{1}^{*}, v_{2}^{*})_{V_{0}^{*}} := 
%    \left\langle v_{1}^{*}, {\cal N}v_{2}^{*} 
%    \right\rangle_{V_{0}^*, V_{0}} 
%    \quad (v_{1}^{*}, v_{2}^{*}\in V_{0}^{*}).  
%    \label{innerVzerostar}
%    \end{align}

\smallskip
We define weak solutions of \sk{the problems \ref{Ptau}, \ref{Ptaueta} as follows.} 

%%%%%%%%%%%%%%%%%%%%%%DefP%%%%%%%%%%%%%%%%%%%%%%%%%%%
\begin{df}
A triplet $(\theta, \mu, \varphi)$ with 
\begin{align}
&\theta \in H^1(0, T; V^*) \cap L^{\infty}(0, T; H) \cap L^2(0, T; V), \label{tausp1}\\
&\displaystyle \pier{t \mapsto  w(t):=\int_{0}^{t} \mu(s)\,ds \ \hbox{ is in }}\
W^{1, \infty}(0, T; V^{*}) \cap L^{\infty}(0, T; V), \label{tausp2}\\
&\varphi \in W^{1, \infty}(0, T; V^{*}) \cap L^{\infty}(0, T; V), \label{tausp3}\\[1mm]
&\beta(\varphi) \in L^{\infty}(0, T; L^q(\Omega)) \label{tausp4}
\end{align}
is called a {\it weak solution} of \ref{Ptau} if 
%$(\theta, \mu, \varphi)$ satisfies 
\begin{align}
&\langle \theta_{t}, v \rangle_{V^{*}, V} 
  + \int_{\Omega}\nabla\theta\cdot\nabla v 
\notag \\ 
  &= \langle f - \varphi_{t}, v \rangle_{V^{*}, V} 
       \qquad \mbox{a.e.\ on}\ (0, T) \quad \mbox{for all}\ v \in V, 
\label{dfPtausol1}   
\\[4mm] 
&\pier{\tau} \langle \varphi_{t}, v \rangle_{V^{*}, V} 
   + (\varphi, v)_{H} 
   + \int_{\Omega}\nabla w \cdot \nabla v 
\notag \\ 
&=  \pier{{} \tau \langle v_{0}, v \rangle_{V^{*}, V}{}} + (\varphi_{0}, v)_{H} 
     \qquad \mbox{a.e.\ on}\ (0, T) \quad \mbox{for all}\ v \in V,  \label{dfPtausol2} 
\\[4mm] 
&\int_{\Omega} \nabla \varphi \cdot \nabla v 
  + \langle \beta(\varphi), v \rangle_{V^{*}, V} 
\notag \\ 
&= \langle w_{t}, v \rangle_{V^{*}, V} 
     + (\theta - \pi(\varphi), v)_{H}
         \qquad \mbox{a.e.\ on}\ (0, T) \quad \mbox{for all}\ v \in V,  \label{dfPtausol3} 
\\[4mm]
&\theta(0) = \theta_{0},\ \varphi(0) = \varphi_{0}
        \quad \mbox{a.e.\ in}\ \Omega. \label{dfPtausol4}
\end{align}
\end{df}

\pcol{Note that the regularity properties in \eqref{tausp1}-\eqref{tausp3} imply that the three functions $\theta, w, \varphi 
$ are all continuous from $[0,T]$ to $H$, so that the initial conditions \eqref{dfPtausol4} make sense.}

%%%%%%%%%%%%%%%%%%%%%%DefP%%%%%%%%%%%%%%%%%%%%%%%%%%%
\sk{\begin{df}
A triplet $(\theta_{\eta}, \mu_{\eta}, \varphi_{\eta})$ with 
\begin{align}
&\theta_{\eta} \in H^1(0, T; V^*) \cap L^{\infty}(0, T; H) \cap L^2(0, T; V), \label{etasp1} 
\\
&\displaystyle t \mapsto  w_{\eta}(t):=\int_{0}^{t} \mu_{\eta}(s)\,ds \ \hbox{ is in }\ 
W^{1, \infty}(0, T; V^{*}) \cap L^{\infty}(0, T; V), \label{etasp2} 
\\
&\varphi_{\eta} \in W^{1, \infty}(0, T; V^{*}) \cap H^{1}(0, T; H)  \cap L^{\infty}(0, T; V), \label{etasp3} 
\\[1mm]
&\beta(\varphi_{\eta}) \in L^{\infty}(0, T; L^q(\Omega)) \label{etasp4}
\end{align}
is called a {\it weak solution} of \ref{Ptaueta} if 
%$(\theta, \mu, \varphi)$ satisfies 
\begin{align}
&\langle (\theta_{\eta})_{t}, v \rangle_{V^{*}, V} 
  + \int_{\Omega}\nabla\theta_{\eta}\cdot\nabla v 
\notag \\ 
  &= \langle f - (\varphi_{\eta})_{t}, v \rangle_{V^{*}, V} 
       \qquad \mbox{a.e.\ on}\ (0, T) \quad \mbox{for all}\ v \in V, 
\label{dfPtauetasol1}   
\\[4mm] 
&\tau \langle (\varphi_{\eta})_{t}, v \rangle_{V^{*}, V} 
   + (\varphi_{\eta}, v)_{H} 
   + \int_{\Omega}\nabla w_{\eta} \cdot \nabla v 
\notag \\ 
&= \tau \langle v_{0}, v \rangle_{V^{*}, V} + (\varphi_{0}, v)_{H} 
     \qquad \mbox{a.e.\ on}\ (0, T) \quad \mbox{for all}\ v \in V,  \label{dfPtauetasol2} 
\\[4mm] 
&\eta((\varphi_{\eta})_{t}, v)_{H} + \int_{\Omega} \nabla \varphi_{\eta} \cdot \nabla v 
  + \langle \beta(\varphi_{\eta}), v \rangle_{V^{*}, V} 
\notag \\ 
&= \langle (w_{\eta})_{t}, v \rangle_{V^{*}, V} 
     + (\theta_{\eta} - \pi(\varphi_{\eta}), v)_{H}
         \qquad \mbox{a.e.\ on}\ (0, T) \quad \mbox{for all}\ v \in V,  \label{dfPtauetasol3} 
\\[4mm]
&\theta(0) = \theta_{0},\ \varphi(0) = \varphi_{0}
        \quad \mbox{a.e.\ in}\ \Omega. \label{dfPtauetasol4}
\end{align}
\end{df}}%

\sk{The following result asserts existence of weak solutions 
to \ref{Ptaueta}: 
\begin{thm}\label{maintheorem1}
Assume that {\rm (C1)-(C3)} hold. 
Then for all $\tau, \eta \in (0, 1]$ 
there exists a weak solution of the problem {\rm \ref{Ptaueta}}.
\end{thm}}%
\sk{\noindent 
Moreover, the \pcol{next statement is concerned with} weak solutions to \ref{Ptau}: 
\begin{thm}\label{maintheorem2}
Assume that {\rm (C1)-(C3)} hold. 
Then \sk{for all $\tau \in (0, 1]$} 
there exists a weak solution of the problem {\rm \ref{Ptau}}.
\end{thm}
}%

\sk{The proofs of Theorems~\ref{maintheorem1} \pcol{and}~\ref{maintheorem2} 
are organized as follows.} 
In Section~\ref{Sec3} we consider \pier{the first approximate problem~\ref{Ptauetaep} \pcol{in terms of the approximation parameter $\ep \in (0, 1]$ for the Yosida regularization of $\beta$. Next, 
the further approximate problem~\ref{Ptauetaeplam} is taken in 
Section~\ref{Sec4} for the parameter $\lambda \in (0, 1]$ intervening in the approximation of the operator $-\Delta$ in the space $H=L^2(\Omega)$. The existence {of solutions to}~\ref{Ptauetaeplam} is shown 
with full details in Section~\ref{Sec4}. Then,}}
in Section~\ref{Sec5} we deduce some \pier{uniform} estimates for \ref{Ptauetaeplam}\pcol{, of course independent of $\lambda$,}
and consequently pass to the limit as $\lambda$ tends to zero. 
In Section~\ref{Sec6} \pcol{we prove the key estimates}  and \pcol{consequently we can} show existence of weak solutions to~\sk{\ref{Ptaueta}} 
by taking the limit in \ref{Ptauetaep} as $\ep \searrow 0$. 
\pcol{Section~\ref{Sec7} is devoted to the asymptotic convergences, first establishing existence of weak solutions to \ref{Ptau}  by taking the limit in \ref{Ptaueta} as $\eta \searrow 0$. Then, for $\eta $
positive or null, we discuss the convergence of the problems
as the inertial coefficient $\tau$ tends to $0$.}

%\vfill\eject % ADDED TO START THE NEXT SESSION AT THE NEW PAGE 

%%==============================================================%%
%%==============                                  ==============%%
%%======                      Section3                    ======%%
%%====                                                      ====%%
%%==                                                          ==%%
%%====   Approximations  ====%%
%%======                                                  ======%%
%%==============                                  ==============%%
%%==============================================================%%

\section{\pier{Approximation in terms of $\ep$}}\label{Sec3}

\pcol{We first introduce 
\begin{align}
&\hbox{the Yosida approximation $\beta_{\ep}$ of $\beta$ on $\mathbb{R}$,}\label{pier1}
\end{align}
and consider the function $\hat{\beta}_{\ep} : \mathbb{R} \to \mathbb{R}$ defined by 
\[
\hat{\beta}_{\ep}(r):=
\displaystyle\inf_{s\in\mathbb{R}}\left\{\frac{1}{2\ep}|r-s|^2+\hat{\beta}(s) \right\} 
\quad \mbox{for}\ r\in \mathbb{R}.
\]
Note that $\hat{\beta}_{\ep}$ is termed as the Moreau--Yosida regularization of $\hat{\beta}$. 
We emphasize that $\hat{\beta}_{\ep}$ satisfies the identity
\[
\hat{\beta}_{\ep}(r)=\frac{1}{2\ep}|r-J_{\ep}^{\beta}(r)|^2 
+ \hat{\beta}(J_{\ep}^{\beta}(r)) \quad \hbox{for all $r\in\mathbb{R}$ and all $\ep \in (0,1]$,}
\]
where $J_{\ep}^{\beta}$ is the resolvent operator of $\beta$  
on $\mathbb{R}$.  
Moreover, it turns out that   
\begin{equation}
\beta_{\ep}(r) = \partial\hat{\beta}_{\ep}(r) = \frac{d}{dr}\hat{\beta}_{\ep}(r), \quad 
0\leq \hat{\beta}_{\ep}(r) \leq \hat{\beta}(r)
\label{rem2}
\end{equation}
for all $r\in\mathbb{R}$ and all $\ep \in (0,1]$ 
(see, e.g., \cite[Theorem 2.9, p.\ 48]{Barbu2010}).}

\pcol{Now, it is worth to point out that the property~\eqref{rem2} and the condition (C1) yield}
\begin{equation}
\pcol{|\beta_{\ep}(r)|^q \leq c_{\beta}(1 + \widehat{\beta}(J_{\ep}^{\beta}(r)))
                        \leq c_{\beta}(1 + \widehat{\beta}_{\ep}(r)) \quad
\hbox{for all $r\in\mathbb{R}$ and all $\ep \in (0, 1]$.}}
\label{rem3}
\end{equation}

\pier{Then, in order to prove existence of weak solutions to \ref{Ptaueta}, 
we consider the approximation
\begin{equation}\label{Ptauetaep}\tag*{(P)$_{\tau, \eta, \ep}$}
\begin{cases}
(\theta_{\ep})_t + (\varphi_{\ep})_t - \Delta\theta_{\ep} = f_{\ep}  
& \mbox{in}\ \Omega\times(0, T), \\[1mm]
\tau (\varphi_{\ep})_{tt} + (\varphi_{\ep})_{t} - \Delta\mu_{\ep} = 0 
& \mbox{in}\ \Omega\times(0, T), \\[1mm]
\sk{\eta(\varphi_{\ep})_{t}} - \Delta\varphi_{\ep} 
+ \beta_{\ep}(\varphi_{\ep}) + \pi(\varphi_{\ep}) = \mu_{\ep} + \theta_{\ep} 
& \mbox{in}\ \Omega\times(0, T), \\[1mm]
\partial_{\nu} \theta_{\ep} =  \partial_{\nu} \mu_{\ep} = \partial_{\nu} \varphi_{\ep} = 0
& \mbox{on}\ \partial\Omega\times(0, T), \\[1mm]
\theta_{\ep}(0) = \theta_{0\ep},\ \varphi_{\ep}(0) = \varphi_{0\ep},\ 
(\varphi_{\ep})_{t}(0) = v_{0\ep} 
& \mbox{in}\ \Omega, 
\end{cases}
\end{equation}
where $\ep \in (0, 1]$, 
\begin{align}
&\hbox{$f_{\ep} \in \pcol{C^0}([0, T]; H)$ for all $\ep \in (0, 1]$, 
$f_{\ep} \to f$ strongly in $L^2(0, T; V^*)$ as $\ep \searrow 0$,}\label{pier2}\\[2mm]
&\hbox{$\theta_{0\ep} \in W$,  
$\varphi_{0\ep} \in \pcol{W \cap H^3(\Omega)}$, $v_{0\ep} \in V$ are smooth approximations of $\theta_{0}, \varphi_{0} , v_{0} $}
\nonumber\\
&\hbox{strongly converging to $\theta_{0}, \varphi_{0} , v_{0} $ in the respective spaces $H, V, V^*$.}
\label{pier3}
\end{align}}%
\pcol{We also require that
\begin{align}
& 0\leq \int_\Omega \hat{\beta}_{\ep}(\varphi_{0\ep}) \leq C \quad \hbox{for all $\ep \in (0, 1]$},
\label{pier3.1}\\ 
&\hbox{\sk{$ m(\varphi_{0\ep}) = m(\varphi_{0})$, \ \ $m(v_{0\ep}) = m(v_{0})\, $ \quad
for all $\ep \in (0, 1]$.}}
\label{pier3.2}
\end{align}}%

\sk{\begin{remark}\label{sk}
\pcol{The reader may wonder whether some approximating families $\{f_{\ep}\}$ and  $\{(\theta_{0\ep}, \varphi_{0\ep}, v_{0\ep})\}$ satisfying \eqref{pier2}-\eqref{pier3.2} exist. Well, it turns out that
for all $\ep \in (0, 1]$ there is a unique $f_{\ep} \in \pcol{C^0([0, T]; H)}$ solving}
\[
\ep\langle (f_{\ep})_{t}, v \rangle_{V^{*}, V} 
+ \ep\int_{\Omega} \nabla f_{\ep} \cdot \nabla v 
+ (f_{\ep}, v)_{H} = \langle f, v \rangle_{V^{*}, V} 
\quad \mbox{a.e.\ on}\ (0, T) \quad \mbox{for all}\ v \in V.
\]
Also, for all $\ep \in (0, 1]$ 
there exists a unique triplet $(\theta_{0\ep}, \varphi_{0\ep}, v_{0\ep})$ 
such that  
\begin{align*}
&\theta_{0\ep} \in W,\ \varphi_{0\ep} \in W \cap H^3(\Omega),\ 
v_{0\ep} \in V, 
\\
&\theta_{0\ep} - \ep\Delta\theta_{0\ep} = \theta_{0}  \quad \pcol{\hbox{a.e. in } \Omega,}
\\
&\varphi_{0\ep} - \ep\Delta\varphi_{0\ep} = \varphi_{0}  \quad \pcol{\hbox{a.e. in } \Omega,}
\\
&(v_{0\ep}, v)_{H} 
+ \ep\int_{\Omega} \nabla v_{0\ep} \cdot \nabla v 
= \langle v_{0}, v \rangle_{V^{*}, V} \quad \mbox{for all}\ v \in V.
\end{align*}
\pcol{Then, it is straightforward to check that, in particular,  $\varphi_{0\ep}$ and $v_{0\ep}$
satisfy \eqref{pier3.1} and \eqref{pier3.2}. As for \eqref{pier3.1}, it is enough to test $\varphi_{0\ep} - \varphi_{0} - \ep\Delta\varphi_{0\ep} =0$ by $\beta_\ep (\varphi_{0\ep})$ and use the property of subdifferential along with the monotonicity of  $\beta_\ep$ to deduce that $\int_\Omega \hat{\beta}_{\ep}(\varphi_{0\ep})\leq\int_\Omega \hat{\beta}_{\ep}(\varphi_{0})$, whence one can conclude by \eqref{rem2} and (C3).}
\end{remark}}

\pcol{\begin{remark}\label{pc}
The condition~\eqref{pier3.2}, while highly practical for subsequent computations, is not essential for obtaining our results. What truly matters in the following analysis is that  $ m(\varphi_{0\ep}) $, $m(v_{0\ep}) $ remain bounded independently 
of $\ep$.
\end{remark}}

%\bigskip%

\pier{We also introduce the maximal monotone operator
\begin{align} 
&- \Delta_{N} \ \hbox{ from its domain $D( - \Delta_{N}) = W $ to $H$,}
 \nonumber\\
&%\qquad
\ \hbox{with \ $ - \Delta_{N} v = -\Delta v$  \ for all $v\in W$.}
\label{pier4} 
\end{align}
Note that 
\begin{align}
	\text{$(- \Delta_N u ,v)_{H}= \int_\Omega \nabla u \cdot \nabla v $ 
	\quad 
	for every $u\in W$ and $v\in V$.} \label{pier5}
\end{align}
Now,} we can define the weak solutions of \ref{Ptauetaep} as follows. 
%%%%%%%%%%%%%%%%%%%%%%DefP%%%%%%%%%%%%%%%%%%%%%%%%%%%
\begin{df}
A triplet $(\theta_{\ep}, \mu_{\ep}, \varphi_{\ep})$ with 
\begin{align}
&\theta_{\ep} \in H^1(0, T; H) \cap L^{\infty}(0, T; V) \cap L^2(0, T; W), \label{epsp1}\\
&\displaystyle \pier{t \mapsto  w_\ep (t):=\int_{0}^{t} \mu_\ep (s)\,ds \ \hbox{ is in }}\ \sk{W^{1, \infty}(0, T; V)} \cap L^2(0, T; W), \label{epsp2}\\
&\varphi_{\ep} \in W^{2, \infty}(0, T; \pier{V^{*}}) \cap W^{1, \infty}(0, T; V), \label{epsp3}\\
&\beta_{\ep}(\varphi_{\ep}) \in L^{\infty}(0, T; H \cap L^q(\Omega)) \label{epsp4}
\end{align}
is called a {\it weak solution} of \ref{Ptauetaep} if 
%$(\theta_{\ep}, \mu_{\ep}, \varphi_{\ep})$ satisfies 
\begin{align}
&(\theta_{\ep})_t + (\varphi_{\ep})_t - \Delta_{N}\theta_{\ep} = f_{\ep}  
       \quad \mbox{a.e.\ in}\ \Omega\times(0, T), \label{dfPtauetaepsol1}   
\\[2mm] 
&\sk{\tau\langle (\varphi_{\ep})_{tt}, v \rangle_{V^{*}, V} 
+ \langle (\varphi_{\ep})_{t}, v \rangle_{V^{*}, V} 
+ \int_{\Omega} \nabla\mu_{\ep} \cdot \nabla v 
= 0}
\notag \\ 
&\hspace{75mm}\sk{\mbox{a.e.\ on}\ (0, T) \quad \mbox{for all}\ v \in V,}  \label{dfPtauetaepsol2} 
\\[2mm] 
&\sk{\eta((\varphi_{\ep})_{t}, v)_{H}} 
  + \int_{\Omega} \nabla \varphi_{\ep} \cdot \nabla v 
  + (\beta_{\ep}(\varphi_{\ep}), v)_{H} 
\notag \\ 
&= %\langle (, v \rangle_{V^{*}, V}  + 
     \pcol{((w_{\ep})_{t} + \theta_{\ep} - \pi(\varphi_{\ep}), v)_{H}}
         \qquad \mbox{a.e.\ on}\ (0, T) \quad \mbox{for all}\ v \in V,  \label{dfPtauetaepsol3} 
\\[4mm] 
&\theta_{\ep}(0) = \theta_{0\ep},\ \varphi_{\ep}(0) = \varphi_{0\ep},\ \sk{(\varphi_{\ep})_{t}(0) = v_{0\ep}}
        \quad \mbox{a.e.\ in}\ \Omega. \label{dfPtauetaepsol4}
\end{align}
\end{df}

The following result is concerned with existence of weak solutions 
to~\ref{Ptauetaep}.
\begin{thm}\label{maintheorem3}
Assume that {\rm (C1)-(C3)} \pcol{and \eqref{pier1}, \eqref{pier2}, \eqref{pier3}} hold. 
Then \sk{for all $\tau, \eta, \ep \in (0, 1]$} 
there exists a weak solution of \pier{the problem}~{\rm \ref{Ptauetaep}}.
\end{thm}

\vspace{10pt}

%%==============================================================%%
%%==============                                  ==============%%
%%======                      Section3                    ======%%
%%====                                                      ====%%
%%==                                                          ==%%
%%====   Existence of discrete solution           ====%%
%%======                                                  ======%%
%%==============                                  ==============%%
%%==============================================================%%

\section{\pier{Approximation in terms of $\lambda$}}\label{Sec4}

In this section we will discuss \pier{another approximate problem~\ref{Ptauetaeplam} based on a new approximation parameter $\lambda \in (0, 1]$. To this aim we introduce the Yosida regularization $(- \Delta_{N})_{\lambda}$ of the operator $(- \Delta_{N})$, which is defined by (cf., e.g., \cite{Barbu2010})
\begin{align}
& (- \Delta_{N})_{\lambda} := \frac 1 \lambda (I - J_\lambda)   , 
\label{pier6}
\end{align}
with $I$ identity operator in $H$ and $J_\lambda:= (I+\lambda (- \Delta_{N}))^{-1}$ resolvent operator. It turns out that, for every $u\in H$, the element $J_\lambda u$ belongs to $ D( - \Delta_{N}) \equiv W$ and solves
\begin{align}
&\hbox{$ J_\lambda u + \lambda (- \Delta_{N}) ( J_\lambda u)= u $ \quad in $H$,}
\label{pier7}
\end{align}
with the consequence that 
\begin{align}
&\hbox{$ \| J_\lambda u \|_H^2  + 2\lambda \|\nabla J_\lambda u\|_H^2 \leq  \|  u \|_H^2  $ \quad for every  $u\in H$.}
\label{pier7-1}
\end{align}
Indeed, it suffices to test \eqref{pier7} by $J_\lambda u $ and apply the Schwartz and Young inequalities. 
In addition, if we rewrite  \eqref{pier7} as 
\begin{align*}
&\hbox{$ (1-\lambda) J_\lambda u + \lambda (I - \Delta_{N}) ( J_\lambda u)= u  $ \quad in $H$,}
\end{align*}
and test by (cf.~\eqref{innerVstar}) $F^{-1} (J_\lambda u)$ we obtain 
\begin{align*}
&\hbox{$ (1-\lambda) \| J_\lambda u \|_{V^*} ^2  + \lambda \| J_\lambda u\|_H^2 \leq  \|  u \|_{V^*}
\| J_\lambda u \|_{V^*}
$,}
\end{align*}
whence
\begin{align}
&\hbox{$ \| J_\lambda u \|_{V^*}   \leq  2 \|  u \|_{V^*}
$ \quad for all  $u\in H$ and $\lambda \in (0,1/2)$.}
\label{pier7-2}
\end{align}
\sk{On the other hand, we can verify that 
\begin{align}
&\hbox{$ \| J_\lambda u \|_V^2 
+ 2\lambda (\|\nabla J_\lambda u\|_H^2 + \|\Delta_{N} J_\lambda u\|_{H}^2) 
\leq  \|  u \|_V^2  $ \quad for every  $u\in V$}
\label{sk7-1}
\end{align}
by multiplying \eqref{pier7} by $-\Delta_{N} J_\lambda u$ 
and by using \eqref{pier7-1}.}}%
\smallskip

\pier{\pcol{About the Yosida regularization} $(- \Delta_{N})_{\lambda}$ \pcol{in \eqref{pier6}} we also point out that
\begin{align}
&\hbox{$(- \Delta_{N})_{\lambda}$ is everywhere defined in $H$} \nonumber\\
&\hbox{and  $(- \Delta_{N})_{\lambda} u  = (- \Delta_{N}) ( J_\lambda u) $ for all $u\in H$ and $\lambda>0$.} \label{delta-n}
\end{align}}%

\pier{We can start to write the approximate problem~\ref{Ptauetaeplam} in terms of  $(\theta_{\lambda}, w_{\lambda}, \varphi_{\lambda})$ as 
a system of abstract differential equations in the space $H$. Here is the system, where the time derivative is still denoted by $(\,\cdot\,)_t$. We have}
\begin{equation}\label{Ptiltauetaeplam}\tag*{${\rm (\tilde{P})}_{\tau, \eta,  \ep, \lambda}$}
\begin{cases}
(\theta_{\lambda})_t + (\varphi_{\lambda})_t 
+ (- \Delta_{N})_{\lambda}\theta_{\lambda} = f_{\ep}  
& \pier{\mbox{in}\ [0, T]}, \\[1mm]
\pier{\tau} (\varphi_{\lambda})_{t} + \varphi_{\lambda}
+ (- \Delta_{N})_{\lambda} w_{\lambda} = \pier{\tau \, v_{0\ep} + \varphi_{0\ep}}
& \pier{\mbox{in}\ [0, T]}, \\[1mm]
\sk{\eta(\varphi_{\lambda})_{t}} + (- \Delta_{N})_{\lambda}\varphi_{\lambda} 
+ \beta_{\ep}(\varphi_{\lambda}) + \pi(\varphi_{\lambda}) 
= (w_{\lambda})_{t} + \theta_{\lambda} 
& \pier{\mbox{in}\ [0, T]}, \\[1mm]
\theta_{\lambda}(0) = \theta_0,\ \varphi_{\lambda}(0) = \varphi_{0\ep}, \ w_{\lambda}(0) = 0,
& %\mbox{in}\ \Omega, 
\end{cases}
\end{equation}
where \pier{$w_{\lambda}$ is related to $\mu_{\lambda}$ by
$w_{\lambda} (t) := \displaystyle \int_{0}^{t} \mu_{\lambda}(s)\,ds$, $t\in [0,T].$}
We define the function $L : H \times H \times H \to H \times H \times H$ by
\begin{align*}
L: 
&\left(
    \begin{array}{c}
      \theta \\
      \varphi \\
      w
    \end{array}
  \right) 
\\
&\mapsto 
\left(
    \begin{array}{c}
      - (- \Delta_{N})_{\lambda}\theta +
      \pier{\frac 1 \tau \bigl( \varphi + (- \Delta_{N})_{\lambda} w 
      - \pier{\tau \, v_{0\ep} - \varphi_{0\ep}}\bigr) }
      \\[3mm]
\pier{\frac 1 \tau \bigl( - \varphi - (- \Delta_{N})_{\lambda} w 
      + \pier{\tau \, v_{0\ep} + \varphi_{0\ep}}\bigr) }
      \\[3mm] 
  \pier{\frac \eta \tau \bigl( - \varphi - (- \Delta_{N})_{\lambda} w 
      + \pier{\tau \, v_{0\ep} + \varphi_{0\ep}}\bigr) }
      + (- \Delta_{N})_{\lambda}\varphi 
      + \beta_{\ep}(\varphi) + \pi(\varphi) - \theta 
    \end{array}
  \right)
\end{align*} 
\pier{and put}
\[
F := \left(
    \begin{array}{c}
      f_{\ep} \\
      0 \\
      0
    \end{array}
  \right) \in \pcol{C^0}([0, T]; H \times H \times H),\ \ 
U_{0} := \left(
    \begin{array}{c}
      \theta_{0\ep} \\
      \varphi_{0\ep} \\
      0
    \end{array}
  \right) \in H \times H \times H. 
\]
We can observe that 
$L : H \times H \times H \to H \times H \times H$ 
is Lipschitz continuous: \pier{indeed, this follows from
(C2) and the Lipschitz continuity of the Yosida approximations 
$(-\Delta_{N})_{\lambda}$ and $\beta_{\ep}$, 
with Lipschitz constants $1/\lambda$ and $1/\ep$, respectively.
Then,} we see from the Cauchy--Lipschitz--Picard theorem that 
there exists a unique solution $U \in C^1([0, T]; H \times H \times H)$ 
of the problem
\begin{equation*}
\begin{cases}
\displaystyle \frac{dU}{dt} = L(U) + F  &\mbox{in}\ [0, T], 
\\[2mm]
U(0) = U_{0}.
\end{cases}
\end{equation*}
Therefore, there exists a unique solution 
$(\theta_{\lambda}, \varphi_{\lambda}, w_{\lambda}) \in (C^1([0, T]; H))^3$ 
of \pier{the problem~\ref{Ptiltauetaeplam}. Now, it is not difficult to check that 
the corresponding triple $(\theta_{\lambda}, \mu_{\lambda}, \varphi_{\lambda}) $ possesses the regularity 
$ C^1([0, T]; H) \times \pcol{C^0}([0, T]; H) \times C^2([0, T]; H)$:  in fact, by using the second equation in~\ref{Ptiltauetaeplam}
one recovers the $C^2$ regularity of $\varphi_{\lambda}$. In addition,
this triple $(\theta_{\lambda}, \mu_{\lambda}, \varphi_{\lambda}) $ solves 
the following evolution problem with equations in $H$:}
\begin{equation}
\label{Ptauetaeplam}\tag*{(P)$_{\tau, \eta, \ep, \lambda}$}
\begin{cases}
(\theta_{\lambda})_t + (\varphi_{\lambda})_t 
+ (- \Delta_{N})_{\lambda}\theta_{\lambda} = f_{\ep}  
& \pier{\mbox{in}\ [0, T]}, \\[1mm]
\pier{\tau}(\varphi_{\lambda})_{tt} + (\varphi_{\lambda})_{t} 
+ (- \Delta_{N})_{\lambda}\mu_{\lambda} = 0 
&  \pier{\mbox{in}\ [0, T]},  \\[1mm]
\sk{\eta(\varphi_{\lambda})_{t}} + (- \Delta_{N})_{\lambda}\varphi_{\lambda} 
+ \beta_{\ep}(\varphi_{\lambda}) + \pi(\varphi_{\lambda}) 
= \mu_{\lambda} + \theta_{\lambda} 
& \pier{\mbox{in}\ [0, T]}, \\[1mm]
\theta_{\lambda}(0) = \theta_{0\ep},\ \varphi_{\lambda}(0) = \varphi_{0\ep},\ 
(\varphi_{\lambda})_{t}(0) = v_{0\ep} .
&% \mbox{in}\ \Omega. 
\end{cases}
\end{equation}

%
%
%\begin{remark}\label{rem4}
Note that, integrating the second equation in \ref{Ptauetaeplam} over $\Omega$ 
\pier{and using the property~\eqref{delta-n}, we find out that}
\begin{equation}
\pier{\tau}\,  m((\varphi_{\lambda})_{tt}(t)) + m((\varphi_{\lambda})_{t}(t)) = 0,
\label{mean-ode}
\end{equation}
\pier{where the mean value $m(\,\cdot\,)$ is defined in \eqref{pier8}.
Then it follows from the initial conditions that 
\begin{equation}
\tau (m(\varphi_{\lambda}))_{t}(t) + m(\varphi_{\lambda}(t)) 
= \tau\, m(v_{0\ep})  +  m(\varphi_{0\ep}) 
= \sk{\tau \,m(v_{0})  +  m(\varphi_{0}) =: m_0}
\label{pier9}
\end{equation}
whence we can obtain
\begin{equation}
m(\varphi_{\lambda}(t)) = \sk{m(\varphi_{0})}e^{-t/\tau } + \sk{m_{0}}(1 - e^{-t/\tau})
\label{pier10}
\end{equation}
and 
\begin{equation}
m((\varphi_{\lambda})_{t}(t)) =  \sk{m(v_{0})}e^{-t/\tau} \quad \hbox{for every } \, t\in [0,T].
\label{pier11}
\end{equation}}%
%\end{remark}
\pier{The next sections contain the derivation of estimates and the procedures for the passages to the limits.
About the constants used for estimates, we point out that in the sequel $\,C\,$ denotes any 
positive constant depending only on data and, in particular,  independent of $\lambda, \, \ep, \, \tau$. 
The value of such generic constants $\,C\,$ may 
change from formula to formula or even within the lines of the same formula. 
Finally, the notation $C_\delta$ denotes a positive constant that additionally depends on the quantity $\delta$.} 

\vspace{10pt}

%%==============================================================%%
%%==============                                  ==============%%
%%======                      Section5                    ======%%
%%====                                                      ====%%
%%==                                                          ==%%
%%====                                                      ====%%
%%======                                                  ======%%
%%==============                                  ==============%%
%%==============================================================%%

\section{Estimates for \ref{Ptauetaeplam} 
and passage to the limit as $\lambda\searrow0$} 
\label{Sec5}

In this section we will establish a priori estimates for \ref{Ptauetaeplam} 
and will prove Theorem~\ref{maintheorem3} 
by passing to the limit in \ref{Ptauetaeplam} as $\lambda \searrow 0$.
\begin{lem}\label{estieplam1}
\sk{For all $\tau, \eta, \ep \in (0, 1]$} 
there exists a constant \sk{$C_{\tau, \eta, \ep} > 0$} such that
\begin{align*}
&\|\varphi_{\lambda}\|_{L^{\infty}(0, T; H)} 
+ \|\theta_{\lambda}\|_{L^{\infty}(0, T; H)} 
\pier{{}+ \|J_{\lambda}\theta_{\lambda}\|_{L^{2}(0, T; V)} 
+ \lambda^{1/2}\|(-\Delta_{N})_{\lambda}\theta_{\lambda}\|_{L^{2}(0, T; H)}}
\\
&
\pier{{}+ \sk{\|(\varphi_{\lambda})_{t}\|_{L^{\infty}(0, T; V^{*})}} 
+ \|(\varphi_{\lambda})_{t}\|_{L^2(0, T; H)}}
 \\
&+ \|J_{\lambda}\varphi_{\lambda}\|_{L^{\infty}(0, T; V)} 
+ \lambda^{1/2}\|(-\Delta_{N})_{\lambda}\varphi_{\lambda}\|_{L^{\infty}(0, T; H)} 
+ \|\widehat{\beta}_{\ep}(\varphi_{\lambda})\|_{L^{\infty}(0, T; L^1(\Omega))}
\leq \sk{C_{\tau, \eta, \ep}}
\end{align*}
for all $\lambda \in (0, 1]$.
\end{lem}
\begin{proof}
It holds that 
\sk{\begin{equation}\label{a1}
\frac{\eta}{2}\frac{d}{dt}\|\varphi_{\lambda}(t)\|_{H}^2 
= \eta((\varphi_{\lambda})_{t}(t), \varphi_{\lambda}(t))_{H}.
\end{equation}}%
\pier{We now test the first equation in \ref{Ptauetaeplam} by $\theta_{\lambda}$. Observing that  
\begin{align}
&((-\Delta_{N})_{\lambda}\theta_{\lambda}(t) , \theta_{\lambda}(t) )_{H} = 
((-\Delta_{N})_{\lambda}\theta_{\lambda}(t),
J_{\lambda}\theta_{\lambda}(t) + \lambda(-\Delta_{N})_{\lambda}\theta_{\lambda}(t) )_{H} 
\notag \\[3mm]
&= \|\nabla J_{\lambda}\theta_{\lambda}(t) \|_{H}^2 
    + \lambda\|(-\Delta_{N})_{\lambda}\theta_{\lambda}(t)\|_{H}^2, 
\label{a1-1}
\end{align}
where the resolvent operator $J_{\lambda}$ of $-\Delta_{N}$ is introduced in \eqref{pier6}-\eqref{delta-n},
it is easy to obtain}
\begin{align}\label{a2}
&\frac{1}{2}\frac{d}{dt}\|\theta_{\lambda}(t)\|_{H}^2 
\pier{{}+\|\nabla J_{\lambda}\theta_{\lambda}(t) \|_{H}^2 
    + \lambda\|(-\Delta_{N})_{\lambda}\theta_{\lambda}(t)\|_{H}^2 }\notag \\[1mm]
&= \pier{{}- ((\varphi_{\lambda})_{t}(t), \theta_{\lambda}(t))_{H}} 
+ (f_{\ep}(t), \pier{\theta_{\lambda}}(t))_{H}.
\end{align}
Multiplying \pier{the difference between the second equation in \ref{Ptauetaeplam} and \eqref{mean-ode}
by 
\[{\cal N}((\varphi_{\lambda})_{t}(t) - m((\varphi_{\lambda})_{t}(t))), \] 
thanks to~\eqref{propN}, \eqref{pier11} and~\eqref{delta-n} we see that 
\begin{align}\label{a3}
&\frac{\tau}{2}\frac{d}{dt}\|(\varphi_{\lambda})_{t}(t) \pier{{}- \sk{m(v_{0})}e^{-t/\tau}}\|_*^2 
+ \|(\varphi_{\lambda})_{t}(t)\pier{{}- \sk{m(v_{0})}e^{-t/\tau}}\|_*^2 \nonumber\\
&{}+ (J_{\lambda}\mu_{\lambda}(t), (\varphi_{\lambda})_{t}(t) - \sk{m(v_{0})}e^{-t/\tau})_H
= 0.
\end{align}
Next, we want to test the third equation in \ref{Ptauetaeplam} 
by 
\[(\varphi_{\lambda})_{t}(t) - m((\varphi_{\lambda})_{t}(t)) = (\varphi_{\lambda})_{t}(t) -  \sk{m(v_{0})}e^{-t/\tau}\] 
and argue similarly as in \eqref{a1-1} for the term with $(-\Delta_{N})_{\lambda}\varphi_{\lambda}(t)$.
Then, recalling \eqref{rem2} we deduce that}
\begin{align}\label{a4}
&\sk{\eta\|(\varphi_{\lambda})_{t}(t)\|_{H}^2} 
+ \frac{1}{2}\frac{d}{dt}
       (\|\nabla J_{\lambda}\varphi_{\lambda}(t) \|_{H}^2 
          + \lambda\|(-\Delta_{N})_{\lambda}\varphi_{\lambda}(t)\|_{H}^2)
+ \frac{d}{dt}\int_{\Omega}\widehat{\beta}_{\ep}(\varphi_{\lambda}(t))
\notag \\
&= \sk{\eta((\varphi_{\lambda})_{t}(t), \sk{m(v_{0})}e^{-t/\tau})_{H}} 
    + \int_{\Omega}\beta_{\ep}(\varphi_{\lambda}(t))\pier{\sk{m(v_{0})}e^{-t/\tau}}
\notag \\
&\,\quad - \pier{( \pi(\varphi_{\lambda}(t)),(\varphi_{\lambda})_{t}(t) - \pier{\sk{m(v_{0})}e^{-t/\tau}})_{H}} 
    + (\mu_{\lambda}(t), (\varphi_{\lambda})_{t}(t) - \pier{\sk{m(v_{0})}e^{-t/\tau}})_{H} 
\notag \\
&\,\quad + (\theta_{\lambda}(t), (\varphi_{\lambda})_{t}(t))_{H} 
    - (\theta_{\lambda}(t), \pier{\sk{m(v_{0})}e^{-t/\tau}})_{H}.
\end{align}
\pier{As for the term containing $\mu_{\lambda}(t)$ above,} the second equation in \ref{Ptauetaeplam} implies that 
\begin{align}\label{a5}
&(\mu_{\lambda}(t), (\varphi_{\lambda})_{t}(t) - \pier{\sk{m(v_{0})}e^{-t/\tau}})_{H} 
\notag \\[1mm]
&= (J_{\lambda}\mu_{\lambda}(t), (\varphi_{\lambda})_{t}(t) - \pier{\sk{m(v_{0})}e^{-t/\tau}})_{H} 
\notag \\
& \,\quad   + \lambda((-\Delta_{N})_{\lambda}\mu_{\lambda}(t), 
                                              (\varphi_{\lambda})_{t}(t) - \pier{\sk{m(v_{0})}e^{-t/\tau}})_{H} 
\notag \\
&= (J_{\lambda}\mu_{\lambda}(t), (\varphi_{\lambda})_{t}(t) - \pier{\sk{m(v_{0})}e^{-t/\tau}})_{H} 
\notag \\
& \,\quad       -\lambda(\pier \tau(\varphi_{\lambda})_{tt}(t) + (\varphi_{\lambda})_{t}(t), 
                                                   (\varphi_{\lambda})_{t}(t) - \pier{\sk{m(v_{0})}e^{-t/\tau}})_{H} 
\notag \\
&= (J_{\lambda}\mu_{\lambda}(t), (\varphi_{\lambda})_{t}(t) - \pier{\sk{m(v_{0})}e^{-t/\tau}})_{H} 
    - \frac{\lambda\pier\tau}{2}\frac{d}{dt}\|(\varphi_{\lambda})_{t}(t)\|_{H}^2 
    - \lambda\|(\varphi_{\lambda})_{t}(t)\|_{H}^2,
\end{align}
\pier{where the equality~\eqref{mean-ode} has been used. 
Now, we replace the related term in the right-hand side of \eqref{a4} with the last line of \eqref{a5},
then we add \eqref{a4} to \eqref{a3} obtaining
\begin{align}\label{a6}
&\frac{\tau}{2}\frac{d}{dt}\|(\varphi_{\lambda})_{t}(t) \pier{{}- \sk{m(v_{0})}e^{-t/\tau}}\|_*^2 
+ \|(\varphi_{\lambda})_{t}(t)\pier{{}- \sk{m(v_{0})}e^{-t/\tau}}\|_*^2 \notag \\
&+ \sk{\eta\|(\varphi_{\lambda})_{t}(t)\|_{H}^2} 
+ \frac{1}{2}\frac{d}{dt}
       (\|\nabla J_{\lambda}\varphi_{\lambda}(t) \|_{H}^2 
          + \lambda\|(-\Delta_{N})_{\lambda}\varphi_{\lambda}(t)\|_{H}^2)
\notag \\
&+ \frac{d}{dt}\int_{\Omega}\widehat{\beta}_{\ep}(\varphi_{\lambda}(t)) + \frac{\lambda\pier\tau}{2}\frac{d}{dt}\|(\varphi_{\lambda})_{t}(t)\|_{H}^2 
    +  \lambda\|(\varphi_{\lambda})_{t}(t)\|_{H}^2
\notag \\
&= \sk{\eta((\varphi_{\lambda})_{t}(t), \sk{m(v_{0})}e^{-t/\tau})_{H}}  
    + \int_{\Omega}\beta_{\ep}(\varphi_{\lambda}(t))\pier{\sk{m(v_{0})}e^{-t/\tau}}
\notag \\
&\,\quad - \pier{( \pi(\varphi_{\lambda}(t)),(\varphi_{\lambda})_{t}(t) - \pier{\sk{m(v_{0})}e^{-t/\tau}})_{H} }
\notag \\
&\,\quad+ (\theta_{\lambda}(t), (\varphi_{\lambda})_{t}(t))_{H} 
    - (\theta_{\lambda}(t), \pier{\sk{m(v_{0})}e^{-t/\tau}})_{H}.
\end{align}
At this point, we sum up \eqref{a1}, \eqref{a2}, \eqref{a6} and note that there is a 
cancellation of the two terms with the scalar product of $\theta_{\lambda}(t)$ and $(\varphi_{\lambda})_{t}(t)$.
We also point out that
\begin{equation}
|\sk{m(v_{0})}e^{-t/\tau}| \leq C \quad \hbox{for every }\, t\in [0,T].
\label{a7}
\end{equation}
All the terms in the right-hand side can be estimated by the Young inequality, taking into \pcol{account~\eqref{rem3}},
the Lipschitz continuity of $\pi$ and the property in \eqref{pier7-1}. Therefore, we can show Lemma \ref{estieplam1} 
by integrating with respect to $t \in [0, T]$, exploit the boundedness of initial data and applying the Gronwall lemma.}
\end{proof}

\begin{lem}\label{estieplam2}
For all \sk{$\tau, \eta, \ep \in (0, 1]$} 
there exists a constant \sk{$C_{\tau, \eta, \ep} > 0$} such that
\begin{align}
&\|\beta_{\ep}(\varphi_{\lambda})\|_{L^{\infty}(0, T; H \cap L^q(\Omega))} \leq \sk{C_{\tau, \eta, \ep}}, 
\label{estieplam2-1}
\\[1mm]
&\|(-\Delta_{N})_{\lambda}w_{\lambda}\|_{L^2(0, T; H)} \leq \sk{C_{\tau, \eta, \ep}} \label{estieplam2-2}
\end{align}
for all $\lambda \in (0, 1]$.
\end{lem}
\begin{proof}
The inequality \eqref{estieplam2-1} \pier{follows 
from Lemma~\ref{estieplam1}, 
the Lipschitz continuity of $\beta_{\ep}$ \pcol{and \eqref{rem3}}. Another consequence of Lemma~\ref{estieplam1} is~\eqref{estieplam2-2}, which can be obtained
by a comparison of terms in} the second equation in \ref{Ptiltauetaeplam}. 
\end{proof}

\begin{lem}\label{estieplam3}
For all \sk{$\tau, \eta, \ep \in (0, 1]$} 
there exists a constant \sk{$C_{\tau, \eta, \ep} > 0$} such that
\begin{equation*}
\|\mu_{\lambda}\|_{\sk{L^{\infty}(0, T; V^*)}} (= \|(w_{\lambda})_{t}\|_{\sk{L^{\infty}(0, T; V^*)}}) 
\leq \sk{C_{\tau, \eta, \ep}}
\end{equation*}
for all $\lambda \in (0, 1]$.
\end{lem}
\begin{proof}
\sk{\pcol{One has to compare the terms of the third equation in~\ref{Ptauetaeplam}. Since (cf.~\eqref{delta-n})
\[
\| (- \Delta_{N})_{\lambda}(\varphi_{\lambda}(t)) \|_{V^*}
\leq \| J_{\lambda}(\varphi_{\lambda}(t)) \|_{V} \quad \hbox{for a.e. } t\in (0,T),
\]
the thesis is a consequence of the estimates in Lemma~\ref{estieplam1} and \eqref{estieplam2-1}, along with the Lipschitz continuity of $\pi$ (see (C2)).}}
\end{proof}

\begin{lem}\label{estieplam4}
For all \sk{$\tau, \eta, \ep \in (0, 1]$} 
there exists a constant \sk{$C_{\tau, \eta, \ep} > 0$} such that
\begin{align*}
\|J_{\lambda}(w_{\lambda})_{t}\|_{\sk{L^{\infty}(0, T; V^*)}} 
\leq \sk{C_{\tau, \eta, \ep}}
\end{align*}
for all $\lambda \in (0, 1/2)$.
\end{lem}
\begin{proof}
\pier{Owing to the inequality~\eqref{pier7-2}, 
it is clear that Lemma \ref{estieplam4} follows immediately from} Lemma \ref{estieplam3}.
\end{proof}

\begin{lem}\label{estieplam5}
For all \sk{$\tau, \eta, \ep \in (0, 1]$} 
there exists a constant \sk{$C_{\tau, \eta, \ep} > 0$} such that
\begin{align}
&(\|J_{\lambda}(\theta_{\lambda})_{t}\|_{L^2(0, T; H)} 
\leq)\ \|(\theta_{\lambda})_{t}\|_{L^2(0, T; H)} \leq \sk{C_{\tau, \eta, \ep}},\ \ 
\|J_{\lambda}\theta_{\lambda}\|_{L^{\infty}(0, T; V)} \leq \sk{C_{\tau, \eta,  \ep}}, 
\label{estieplam5-1}
\\[1mm]
&\|(-\Delta_{N})_{\lambda}\theta_{\lambda}\|_{L^2(0, T; H)} \leq \sk{C_{\tau, \eta, \ep}} \label{estieplam5-2}
\end{align}
for all $\lambda \in (0, 1]$.
\end{lem}
\begin{proof}
Multiplying the first equation in \ref{Ptauetaeplam} 
by $(\theta_{\lambda})_{t}(t)$ \pier{and arguing as in the proof of Lemma~\ref{estieplam1},   
we deduce that
\begin{align*}
&\|(\theta_{\lambda})_{t}(t)\|_{H}^2 
+ \frac{1}{2}\frac{d}{dt}
       (\|\nabla J_{\lambda}\theta_{\lambda}(t) \|_{H}^2 
          + \lambda\|(-\Delta_{N})_{\lambda}\theta_{\lambda}(t)\|_{H}^2)
\notag \\
&= (f_\ep (t) - (\varphi_{\lambda})_{t}(t), (\theta_{\lambda})_{t}(t))_H . 
\notag
\end{align*}
Then, by the Young inequality and the conclusion of Lemma~\ref{estieplam1}, integrating with respect to time 
yields \eqref{estieplam5-1}. 
Moreover, the inequality \eqref{estieplam5-2} comes from a subsequent comparison of terms in  
the first equation of \ref{Ptauetaeplam}.}
\end{proof}

\begin{lem}\label{estieplam6}
For all \sk{$\tau, \eta, \ep \in (0, 1]$} 
there exists a constant \sk{$C_{\tau, \eta, \ep} > 0$} such that
\begin{equation*}
(\|J_{\lambda}w_{\lambda}\|_{L^{\infty}(0, T; H)} \leq)\ 
\|w_{\lambda}\|_{L^{\infty}(0, T; H)} \leq \sk{C_{\tau, \eta, \ep}}
\end{equation*}
for all $\lambda \in (0, 1]$.
\end{lem}
\begin{proof}
\pier{The estimate can be obtained by testing the third equation in \ref{Ptauetaeplam} by $w_{\lambda}$ and  
integrating over $(0, t)$, for $t \in [0, T]$. Then, most of terms can be easily treated with the help of 
Lemma~\ref{estieplam1} and \eqref{estieplam2-1}, while we point that the term
\[ \int_0^t ((-\Delta_{N})_{\lambda}\varphi_{\lambda}(s), w_\lambda (s))_H ds =
 \int_0^t (\varphi_{\lambda}(s),(-\Delta_{N})_{\lambda} w_\lambda (s))_H ds \]
 is under control, thanks to \eqref{estieplam2-2}. By the Gronwall lemma one ends the proof.}
 \end{proof}

\begin{lem}\label{estieplam7}
For all \sk{$\tau, \eta, \ep \in (0, 1]$} 
there exists a constant \sk{$C_{\tau, \eta, \ep} > 0$} such that
\begin{equation*}
\|J_{\lambda}w_{\lambda}\|_{L^{\infty}(0, T; V)} \leq \sk{C_{\tau, \eta, \ep}}
\end{equation*}
for all $\lambda \in (0, 1]$.
\end{lem}
\begin{proof}
\pier{From the second equation in \ref{Ptiltauetaeplam} we see that \[(- \Delta_{N})_{\lambda} w_{\lambda} = \tau \, 
v_{0\ep} + \varphi_{0\ep} - \tau (\varphi_{\lambda})_{t} - \varphi_{\lambda}, \] with the right-hand side 
bounded in $L^\infty (0,T; V^*)$ by \pcol{$C_{\tau, \eta, \ep}$} (cf. Lemma~\ref{estieplam1}). Then, multiplying 
by $J_{\lambda}w_{\lambda}$ and invoking Lemma~\ref{estieplam6}, 
by the Young inequality we plainly obtain Lemma~\ref{estieplam7}.}
\end{proof}

\begin{lem}\label{estieplam8}
For all \sk{$\tau, \eta, \ep \in (0, 1]$} 
there exists a constant \sk{$C_{\tau, \eta, \ep} > 0$} such that
\begin{equation*}
\|(\varphi_{\lambda})_{tt}\|_{L^{\infty}(0, T; V^{*})} 
+ \|J_{\lambda}(\varphi_{\lambda})_{t}\|_{L^{\infty}(0, T; V)} 
\leq \sk{C_{\tau, \eta, \ep}} 
\end{equation*}
for all $\lambda \in (0, 1]$.
\end{lem}
\begin{proof}
\pier{\pcol{We make} a comparison of the terms in the third equation of \ref{Ptauetaeplam}\pcol{: thus,} as
$ \varphi_{\lambda} \in  C^2([0, T]; H)$ and 
$\theta_{\lambda}\in  C^1([0, T]; H) $, it turns out that
$\mu_{\lambda} $ possesses the regularity $ W^{1,\infty} (0, T; H) $, at least,  due to the Lipschitz continuity of $\beta_\ep $ and $\pi$. Hence, by a comparison in the second equation of \ref{Ptauetaeplam} we gain the same regularity for 
$(\varphi_{\lambda})_{tt}$. Then we can differentiate both equations with respect to time obtaining
\begin{align}\label{b3}
&\sk{\eta(\varphi_{\lambda})_{tt}(t)} + (-\Delta_{N})_{\lambda}(\varphi_{\lambda})_{t}(t) 
+ \beta_{\ep}'(\varphi_{\lambda}(t))(\varphi_{\lambda})_{t}(t) 
+ \pi'(\varphi_{\lambda}(t))(\varphi_{\lambda})_{t}(t) 
\notag \\%[2mm]
&= (\mu_{\lambda})_{t}(t) + (\theta_{\lambda})_{t}(t),
\\[2mm]
\label{b1}
&\tau (\varphi_{\lambda})_{ttt}(t) + (\varphi_{\lambda})_{tt}(t) 
+ (-\Delta_{N})_{\lambda}(\mu_{\lambda})_{t}(t) = 0, 
\end{align}
for a.e. $t\in (0,T)$. First, we multiply the difference between \eqref{b1} and the time derivative of \eqref{mean-ode} 
by ${\cal N}((\varphi_{\lambda})_{tt}(t) - m((\varphi_{\lambda})_{tt}(t)) )$ and observe that 
$m((\varphi_{\lambda})_{tt}(t)) =  - \pier{{}\sk{m(v_{0})}(e^{-t/\tau}\!/\tau)} $ thanks to \eqref{pier11}.
Then, with the help of \eqref{dadefN}-\eqref{danormaVp} and \eqref{delta-n} we infer that
\begin{align}\label{b2}
&\frac{\tau}{2}\frac{d}{dt}\|(\varphi_{\lambda})_{tt}(t) + \sk{m(v_{0})}(e^{-t/\tau}\!/\tau)  \|_*^2 
+ \|(\varphi_{\lambda})_{tt}(t)
+ \sk{m(v_{0})}(e^{-t/\tau}\!/\tau)\|_*^2 \notag\\
&{}+ (J_{\lambda}(\mu_{\lambda})_{t}(t), (\varphi_{\lambda})_{tt}(t) + \pier{\sk{m(v_{0})}(e^{-t/\tau}\!/\tau)})_{H} 
= 0.
\end{align}
Next, we test \eqref{b3} by 
\[
J_{\lambda}((\varphi_{\lambda})_{tt}(t) + \pier{\sk{m(v_{0})}(e^{-t/\tau}\!/\tau)})
= J_{\lambda}((\varphi_{\lambda})_{tt}(t) ) + \pier{\sk{m(v_{0})}(e^{-t/\tau}\!/\tau)}, \]
and note that $J_\lambda$ defined in \eqref{pier7} is a positive and self-adjoint operator. Then, we obtain 
\begin{align}\label{b4}
&\sk{\eta\|J_{\lambda}^{1/2}(\varphi_{\lambda})_{tt}(t)\|_{H}^2}  
+ \frac{1}{2}\frac{d}{dt}\|\nabla J_{\lambda}(\varphi_{\lambda})_{t}(t)\|_{H}^2
\notag \\[1mm]
&= (J_{\lambda}(\mu_{\lambda})_{t}(t), (\varphi_{\lambda})_{tt}(t) + \pier{\sk{m(v_{0})}(e^{-t/\tau}\!/\tau)})_{H} 
%\notag \\
%&\,\quad 
-  \sk{\eta(J_{\lambda}(\varphi_{\lambda})_{tt}(t), 
                                        m(v_{0})(e^{-t/\tau}\!/\tau))_{H}}
\notag \\
&\,\quad+ ((\theta_{\lambda})_{t}(t), 
                                    J_{\lambda}((\varphi_{\lambda})_{tt}(t) )+ \pier{\sk{m(v_{0})}(e^{-t/\tau}\!/\tau)})_{H} 
\notag \\
&\,\quad- (\beta_{\ep}'(\varphi_{\lambda}(t))(\varphi_{\lambda})_{t}(t), 
                                     J_{\lambda}((\varphi_{\lambda})_{tt}(t)) + \pier{\sk{m(v_{0})}(e^{-t/\tau}\!/\tau)})_{H} 
\notag \\
&\,\quad- (\pi'(\varphi_{\lambda}(t))(\varphi_{\lambda})_{t}(t), 
                                    J_{\lambda}((\varphi_{\lambda})_{tt}(t)) + \pier{\sk{m(v_{0})}(e^{-t/\tau}\!/\tau)})_{H}.
\end{align}
About initial values, from the second and third equations plus initial conditions in~\ref{Ptauetaeplam} it follows that 
\begin{align*}
&\tau (\varphi_{\lambda})_{tt}(0) = - v_{0\ep} - (-\Delta_{N})_{\lambda}\mu_{\lambda}(0),
\\
&\mu_{\lambda}(0) 
= -\theta_{0\ep} + \sk{\eta v_{0\ep}} + (-\Delta_{N})_{\lambda}\varphi_{0\ep} 
   + \beta_{\ep}(\varphi_{0\ep}) + \pi(\varphi_{0\ep}).
\end{align*}
Now, \pcol{recalling~\eqref{sk7-1}, \eqref{delta-n} and~\eqref{pier3}} we note that 
\begin{align*}
&\|(-\Delta_{N})_{\lambda}\mu_{\lambda}(0)\|_{V^*} 
\leq\|J_{\lambda}\mu_{\lambda}(0)\|_{V} 
\leq \|\mu_{\lambda}(0)\|_{V}, 
\\
&\|(-\Delta_{N})_{\lambda}\varphi_{0\ep}\|_{V} 
= \|J_{\lambda}(-\Delta_{N}\varphi_{0\ep})\|_{V} 
\leq \|-\Delta_{N}\varphi_{0\ep}\|_{V}, 
\\
&\|J_{\lambda}v_{0\ep}\|_{V} \leq \|v_{0\ep}\|_{V}
\end{align*}
for all \sk{$\ep, \lambda \in (0, 1]$.}}
Hence, \pcol{we can add \eqref{b2} and \eqref{b4}, then integrate with respect to $t \in (0, T)$. 
Using the bounds on initial values, the estimates from Lemma~\ref{estieplam1} and
\eqref{estieplam5-1}, and the property $\|J_{\lambda} v\|_{H} \leq
\|J_{\lambda}^{1/2}v\|_{H}$ for all $v\in H$,
by the Young inequality and the Gronwall lemma we infer that}
for all \sk{$\tau, \eta, \ep \in (0, 1]$} 
there exists a constant \sk{$C_{1, \tau, \eta, \ep} > 0$} such that 
\begin{equation}\label{b5}
\|(\varphi_{\lambda})_{tt}\|_{L^{\infty}(0, T; V^{*})} 
+ \|J_{\lambda}^{1/2}(\varphi_{\lambda})_{tt}\|_{L^2(0, T; H)} 
+ \|\nabla J_{\lambda}(\varphi_{\lambda})_{t}(t)\|_{L^{\infty}(0, T; H)} 
\leq \sk{C_{1, \tau, \eta, \ep}}
\end{equation}
for all $\lambda \in (0, 1]$. 
Moreover, we deduce from \eqref{b5} that 
\begin{align}\label{b6}
&\|J_{\lambda}(\varphi_{\lambda})_{t}(t)\|_{H} 
%\notag \\
%&
\leq \|J_{\lambda}^{1/2}(\varphi_{\lambda})_{t}(t)\|_{H} 
\notag \\
&\leq \int_{0}^{t}\|J_{\lambda}^{1/2}(\varphi_{\lambda})_{tt}(s)\|_{H}\,ds 
       + \|J_{\lambda}^{1/2}v_{0\ep}\|_{H}  
\notag \\
&\leq T^{1/2}\|J_{\lambda}^{1/2}(\varphi_{\lambda})_{tt}\|_{L^2(0, T; H)} 
       + \|v_{0\ep}\|_{H} 
\notag \\
&\leq C_{1, \tau, \eta, \ep}T^{1/2} + \|v_{0\ep}\|_{H}.
\end{align}
Therefore, \pcol{by \eqref{b5} and \eqref{b6} we easily conclude the proof of Lemma~\ref{estieplam8}.}
\end{proof}
\sk{\begin{lem} \label{estieplam9}
For all \sk{$\tau, \eta, \ep \in (0, 1]$} 
there exists a constant \sk{$C_{\tau, \eta, \ep} > 0$} such that
\[
\|J_{\lambda}\mu_{\lambda}\|_{L^{\infty}(0, T; V)} \leq C_{\tau, \eta, \ep}
\]
for all $\lambda \in (0, 1]$.
\end{lem}}%
\sk{\begin{proof}
\pcol{Due to \eqref{mean-ode}, we can apply the operator $\cal N$ defined in \eqref{defN} to the second equation in~\ref{Ptauetaeplam}, and find out that 
\[
{\cal N} ({\tau}(\varphi_{\lambda})_{tt} + (\varphi_{\lambda})_{t} ) =
- J_{\lambda}\mu_{\lambda} + m( J_{\lambda}\mu_{\lambda}) \quad \hbox{a.e. in } (0,T). 
\]
Hence, owing to Lemmas \ref{estieplam1} and \ref{estieplam8}, by comparison of terms we infer that 
\begin{align}\label{bb1}
\|J_{\lambda}\mu_{\lambda} - m( J_{\lambda}\mu_{\lambda})\|_{L^{\infty}(0, T; V)} 
\leq C_{\tau, \eta, \ep}
\end{align}
for all $\lambda \in (0, 1]$. 
On the other hand, applying $J_\lambda$ to the third equation in~\ref{Ptauetaeplam} leads~to 
\begin{align*}
J_{\lambda}\mu_{\lambda} 
= \eta J_{\lambda}(\varphi_{\lambda})_{t} 
   -\Delta_{N}J_{\lambda}J_{\lambda}\varphi_{\lambda} 
   + J_{\lambda}\beta_{\ep}(\varphi_{\lambda}) 
   + J_{\lambda}\pi(\varphi_{\lambda}) 
   - J_{\lambda}\theta_{\lambda} \quad \hbox{a.e. in } (0,T).  
\end{align*}
Then, taking the mean value of all terms,
from Lemmas~\ref{estieplam1}, \ref{estieplam2}, \ref{estieplam8} it follows  that 
for all $\tau, \eta, \ep \in (0, 1]$ 
there exists a constant $C_{\tau, \eta, \ep} > 0$ 
such that 
\begin{align}\label{bb2}
\|m(J_{\lambda}\mu_{\lambda})\|_{L^{\infty}(0, T)} 
\leq C_{\tau, \eta, \ep}
\end{align}
for all $\lambda \in (0, 1]$. 
Therefore, \eqref{bb1}, \eqref{bb2} allow us to prove Lemma~\ref{estieplam9}.}
\end{proof}}
\begin{prth3.1}
Let \sk{$\tau, \ep \in (0, 1]$.} 
Then, owing to Lemmas \ref{estieplam1}-\ref{estieplam9}, 
the compact embeddings 
$V \hookrightarrow H \hookrightarrow V^*$, 
the Aubin--Lions lemma, 
there exist some functions 
$\theta_{\ep}$, $w_{\ep}$, $\mu_{\ep}$, $\varphi_{\ep}$, $\xi_{\ep}$ 
fulfilling 
\begin{align*}
&\theta_{\ep} \in H^1(0, T; H) \cap L^{\infty}(0, T; V) \cap L^2(0, T; W), \\
&w_{\ep} \in \pcol{W^{1, \infty}(0, T; V) \cap L^2(0, T; W),} \\
&\mu_{\ep} \in \sk{L^\infty (0, T; V)}, \\
&\varphi_{\ep} \in \sk{W^{2, \infty}(0, T; V^{*})} \cap W^{1, \infty}(0, T; V), \\
&\xi_{\ep} \in L^{\infty}(0, T; H \cap L^q(\Omega)), 
\end{align*}
and
\begin{align}
&\theta_{\lambda}\pcol{, J_{\lambda}\theta_{\lambda}}  \to \theta_{\ep} 
\quad \mbox{weakly in}\ H^1(0, T; H), 
\label{wc1ep} \\
&J_{\lambda}\theta_{\lambda} \to \theta_{\ep} 
\quad \mbox{weakly* in}\ L^{\infty}(0, T; V), 
\label{wc2ep} \\
&J_{\lambda}\theta_{\lambda} \to \theta_{\ep} 
\quad \mbox{strongly in}\ \pcol{C^0}([0, T]; H), 
\label{st1ep} \\
&(-\Delta_{N})_{\lambda}\theta_{\lambda} \to -\Delta_{N}\theta_{\ep} 
\quad \mbox{weakly in}\ L^2(0, T; H), 
\label{wc3ep} \\[2mm]
&\varphi_{\lambda} \to \varphi_{\ep} 
\quad \mbox{weakly* in}\ \sk{W^{2, \infty}(0, T; V^*)} \cap L^{\infty}(0, T; H), 
\label{wc4ep} \\
&\varphi_{\lambda} \to \varphi_{\ep} 
\quad \mbox{weakly in}\ H^1(0, T; H), 
\label{wc5ep} \\
&J_{\lambda}\varphi_{\lambda} \to \varphi_{\ep} 
\quad \mbox{weakly* in}\ W^{1, \infty}(0, T; V), 
\label{wc6ep} \\
&J_{\lambda}\varphi_{\lambda} \to \varphi_{\ep} 
\quad \mbox{strongly in}\ \pcol{C^0}([0, T]; H), 
\label{st2ep} \\
&\beta_{\ep}(\varphi_{\lambda}) \to \xi_{\ep} 
\quad \mbox{weakly* in}\ L^{\infty}(0, T; H \cap L^q(\Omega)), 
\label{wc7ep} \\[2mm]
&w_{\lambda} \to w_{\ep} 
\quad \mbox{weakly* in}\ L^{\infty}(0, T; H), 
\label{wc8ep} \\
&(w_{\lambda})_{t}\ (= \mu_{\lambda}) \to \mu_{\ep}\ \pcol{(=(w_{\ep})_{t})}
\quad \mbox{weakly* in}\ \sk{L^{\infty} (0, T; V^{*})}, 
\label{wc9ep} \\
&\sk{\pcol{J_{\lambda}w_{\lambda} \to w_{\ep} \ \hbox{ and }}\  J_{\lambda}\mu_{\lambda} \to \mu_{\ep} 
\quad \mbox{weakly* in}\ \sk{L^{\infty} (0, T; V)},} 
\label{wc10ep} \\
&(-\Delta_{N})_{\lambda}w_{\lambda} = -\Delta_{N}w_{\ep} 
\quad \mbox{weakly in}\ L^2(0, T; H), 
\label{wc11ep} \\
%&J_{\lambda}w_{\lambda} \to w_{\ep} 
%\quad \mbox{weakly* in}\ L^{\infty}(0, T; V), 
%\label{wc12ep} \\
&J_{\lambda}w_{\lambda} \to w_{\ep} 
\quad \mbox{strongly in}\ \pcol{C^0}([0, T]; H) 
\label{st3ep}
\end{align}
\pcol{along a subsequence of} $\lambda = \lambda_{j} \searrow 0$. 
It follows from \eqref{st2ep} and Lemma \ref{estieplam1} that 
\begin{equation}\label{st4ep}
\varphi_{\lambda} 
= J_{\lambda}\varphi_{\lambda} 
   + \lambda^{1/2}\cdot\lambda^{1/2}(-\Delta_{N})_{\lambda}\varphi_{\lambda} 
\to \varphi_{\ep}  
\quad \mbox{strongly in}\ L^{\infty}(0, T; H)
\end{equation}
 $\lambda = \lambda_{j} \searrow 0$. 
Therefore, thanks to \eqref{wc1ep}-\eqref{st4ep} \pcol{and
the Lipschitz continuity of $\beta_{\ep}$ and $\pi$, it turns out that $\xi_\ep 
= \beta_{\ep}(\varphi_{\ep})$ and the passage to the limit in~\ref{Ptauetaeplam} as $\lambda = \lambda_{j} \searrow 0$
enables us to conclude that  $(\theta_{\ep},\mu_{\ep}, \varphi_{\ep})$
is a weak solutions of~\ref{Ptauetaep}.}
 \qed 
\end{prth3.1}

\pcol{\begin{remark}
\label{reg}
Note that the found solution  $(\theta_{\ep},\mu_{\ep}, \varphi_{\ep})$ also satisfies
 $\varphi_{\ep} \in L^\infty(0,T;W)$. Indeed, from \eqref{dfPtauetaepsol3} it follows that 
\begin{align*}
&\int_{\Omega} \nabla \varphi_{\ep} \cdot \nabla v 
=  (\mu_{\ep} + \theta_{\ep} - \pi(\varphi_{\ep}) - \eta(\varphi_{\ep})_{t}- \beta_{\ep}(\varphi_{\ep}), v)_{H} \quad \mbox{a.e.\ on}\ (0, T) 
\end{align*}
for all $ v \in V,$ with the term $\mu_{\ep} + \theta_{\ep} - \pi(\varphi_{\ep}) - \eta(\varphi_{\ep})_{t}- \beta_{\ep}(\varphi_{\ep}) $ on the right-hand side that is in $L^\infty(0,T;H)$, at least. Then, the assertion follows from well-known elliptic regularity results. 
\end{remark}}

\vspace{10pt}

%%==============================================================%%
%%==============                                  ==============%%
%%======                      Section6                    ======%%
%%====                                                      ====%%
%%==                                                          ==%%
%%====                                                      ====%%
%%======                                                  ======%%
%%==============                                  ==============%%
%%==============================================================%%

\section{Estimates for \ref{Ptauetaep} and passage to the limit as $\ep\searrow0$} 
\label{Sec6}

In this section we will establish a priori estimates for \ref{Ptauetaep} 
and will prove Theorem~\ref{maintheorem1} 
by passing to the limit in \ref{Ptauetaep} as $\ep \searrow 0$.

\begin{lem}\label{estiep1}
There exists a constant $C > 0$ such that 
\begin{align*}
&\|\theta_{\ep}\|_{L^{\infty}(0, T; H)} 
+ \|\nabla\theta_{\ep}\|_{L^2(0, T; H)} 
\\
&+ \sk{\eta^{1/2}\|(\varphi_{\ep})_{t}\|_{L^2(0, T; H)}} 
+ \sk{\tau^{1/2}\|(\varphi_{\ep})_{t}\|_{L^{\infty}(0, T; V^{*})}} 
+ \pcol{\|(\varphi_{\ep})_{t}\|_{L^2(0, T; V^{*})}} 
\\
&+ \|\nabla\varphi_{\ep}\|_{L^{\infty}(0, T; H)} 
+ \|\widehat{\beta}_{\ep}(\varphi_{\ep})\|_{L^{\infty}(0, T; L^1(\Omega))} 
\leq C
\end{align*}
for all \sk{$\tau, \eta, \ep \in (0, 1]$.}
\end{lem}
\begin{proof}
\sk{\pcol{Testing} \eqref{dfPtauetaepsol1} by $\theta_{\ep}(t)$, 
taking $v = \mathcal{N}((\varphi_{\ep})_{t}(t) - m(v_{0})e^{-t/\tau})$ 
in \eqref{dfPtauetaepsol2}, 
choosing $v = (\varphi_{\ep})_{t}(t) - m(v_{0})e^{-t/\tau}$ 
in \eqref{dfPtauetaepsol3}\pcol{, finally summing the three equalities}
yield that} 
\begin{align}\label{c1}
&\frac{1}{2}\frac{d}{dt}\|\theta_{\ep}(t)\|_{H}^2 
+ \|\nabla\theta_{\ep}(t)\|_{H}^2 
\notag \\
&+ \sk{\frac{\tau}{2}\frac{d}{dt}\|(\varphi_{\ep})_{t}(t) - m(v_{0})e^{-t/\tau}\|_*^2 
+ \|(\varphi_{\ep})_{t}(t) - m(v_{0})e^{-t/\tau}\|_*^2} 
\notag \\
&+ \sk{\eta\|(\varphi_{\ep})_{t}(t)\|_{H}^2} 
+ \frac{1}{2}\frac{d}{dt}\|\nabla\varphi_{\ep}(t)\|_{H}^2 
+ \frac{d}{dt}\int_{\Omega} \widehat{\beta}_{\ep}(\varphi_{\ep}(t))
\notag \\[2mm]
&= \langle f_{\ep}(t), \theta_{\ep}(t) \rangle_{V^{*}, V} 
          - \langle (\varphi_{\ep})_{t}(t) - \sk{m(v_{0})e^{-t/\tau}}, \pi(\varphi_{\ep}(t)) \rangle_{V^{*}, V} 
\notag \\
&\,\quad 
+ \sk{\eta((\varphi_{\ep})_{t}(t), m(v_{0})e^{-t/\tau})_{H}} \pcol{{}+{}} (\theta_{\ep}(t), \sk{m(v_{0})e^{-t/\tau}})_{H}
\notag \\
&\,\quad 
+ \int_{\Omega}\beta_{\ep}(\varphi_{\ep}(t))\sk{m(v_{0})e^{-t/\tau}}  . 
\end{align}
\pcol{Concerning the right-hand side, it is not difficult to see that} \sk{there exists a constant $C_{1}> 0$} such that 
\begin{align}\label{c2}
&\langle f_{\ep}(t), \theta_{\ep}(t) \rangle_{V^{*}, V} 
          - \langle (\varphi_{\ep})_{t}(t) - \sk{m(v_{0})e^{-t/\tau}}, \pi(\varphi_{\ep}(t)) \rangle_{V^{*}, V} 
\notag \\
&= \langle f_{\ep}(t), \theta_{\ep}(t) \rangle_{V^{*}, V} 
    - \langle (\varphi_{\ep})_{t}(t) - \sk{m(v_{0})e^{-t/\tau}}, 
                        \pi(\varphi_{\ep}(t)) - m(\pi(\varphi_{\ep}(t))) 
                                                                      \rangle_{V^*, V}  
\notag \\
&\leq \sk{\frac{1}{2}\|\theta_{\ep}(t)\|_{V}^2 \pcol{{} + \frac{1}{2}\|f_{\ep}(t)\|_{V^{*}}^2}
+ \frac{1}{2}\|(\varphi_{\ep})_{t}(t) - m(v_{0})e^{-t/\tau}\|_*^2 
+ C_{1}\|\nabla\varphi_{\ep}(t)\|_{H}^2 }
\end{align}
for all \sk{$\tau, \eta, \ep \in (0, 1]$} and a.a.\ $t \in (0, T)$. \pcol{In view of \eqref{a7}, the third and fourth terms on the right-hand side of \eqref{c1} can be easily treated with the Young inequality. The last term in~\eqref{c1} is under control due \pcol{to~\eqref{rem3}}. If now we integrate with respect to time, we also obtain the contribution due to the initial values
\[
\frac{1}{2}\|\theta_{0\ep}\|_{H}^2 + \frac{\tau}{2}\|v_{0\ep} - m(v_{0})\|_*^2 
+ \frac{1}{2}\|\nabla\varphi_{0\ep}\|_{H}^2 
+ \int_{\Omega} \widehat{\beta}_{\ep}(\varphi_{0\ep}),
\]
which is however bounded independently of $\ep$ due to \eqref{pier3} and \eqref{pier3.1}.  
Therefore, the proof of Lemma~\ref{estiep1} can be completed by applying the Gronwall lemma.} 
\end{proof}

\begin{lem}\label{estiep2}
There exists a constant $C > 0$ such that 
\[
\|\varphi_{\ep}\|_{\pcol{L^{\infty}(0, T; V)}} \leq C
\]
for all \sk{$\tau, \eta, \ep \in (0, 1]$.}
\end{lem}
\begin{proof}
Since \pcol{(cf.~\eqref{pier10})}
\begin{equation*}
m(\varphi_{\ep}(t)) = \sk{m(\varphi_{0})}e^{-t/\tau } + \sk{m_{0}}(1 - e^{-t/\tau}),
\end{equation*}
we can obtain Lemma \ref{estiep2} 
by Lemma \ref{estiep1} and the Poincar\'e--Wirtinger inequality. 
\end{proof}

\begin{lem}\label{estiep3}
There exists a constant $C > 0$ such that 
\[
\|\beta_{\ep}(\varphi_{\ep})\|_{L^{\infty}(0, T; L^q(\Omega))} \leq C
\]
for all \sk{$\tau, \eta, \ep \in (0, 1]$.}
\end{lem}
\begin{proof}
\pcol{The thesis holds true due to} Lemma~\ref{estiep1} \pcol{and~\eqref{rem3}}. 
\end{proof}

\begin{lem}\label{estiep4}
There exists a constant $C > 0$ such \pcol{that 
%\[
%\|\mu_{\ep}\|_{L^{\infty}(0, T; V^*)} (= \|(w_{\ep})_{t}\|_{L^{\infty}(0, T; V^*)}) \leq C
%\]
%\sk{for all $\tau, \ep \in (0, 1]$ and all $0 < \eta \leq C \tau$, whereas 
\begin{equation}\label{estiep4-1}
\|\mu_{\ep}\|_{L^{2}(0, T; V^*)} (= \|(w_{\ep})_{t}\|_{L^{2}(0, T; V^*)}) \leq C
\end{equation}
for all $\tau, \eta, \ep \in (0, 1]$.} 
Also, for all $\tau \in (0, 1]$ there exists a constant $C_{\tau} > 0$ such that 
\begin{equation}\label{estiep4-2}
\|\mu_{\ep}\|_{L^{\infty}(0, T; V^*)} (= \|(w_{\ep})_{t}\|_{L^{\infty}(0, T; V^*)}) \leq C_{\tau}
\end{equation}
for all $\eta, \ep \in (0, 1]$.
\end{lem}
\begin{proof}
\sk{We can \pcol{make a comparison of the terms in \eqref{dfPtauetaepsol3} and use the conclusions of 
Lemmas~\ref{estiep1}-\ref{estiep3}, along with (C2) and the continuous embedding $L^q(\Omega)\hookrightarrow V^*$, to prove Lemma~\ref{estiep4}.}} 
\end{proof}

\begin{lem}\label{estiep5}
There exists a constant $C > 0$ such that 
\[
\|w_{\ep}\|_{L^{\infty}(0, T; H)\pcol{{}\cap L^2(0,T;V)}} \leq C
\]
for all \sk{$\tau, \eta, \ep \in (0, 1]$.}
\end{lem}
\begin{proof}
We choose $v = w_{\ep}$ in \eqref{dfPtauetaepsol3} to have that
\begin{align}\label{d1}
&\frac{1}{2}\frac{d}{dt}\|w_{\ep}(t)\|_{H}^2 
\notag \\[2mm]
&= \sk{{}- (\theta_{\ep}(t), w_{\ep}(t))_{H} 
    + \eta((\varphi_{\ep})_{t}(t), w_{\ep}(t))_{H} 
    + \int_{\Omega} \nabla\varphi_{\ep}(t) \cdot \nabla w_{\ep}(t)} 
\notag \\
&\,\quad\sk{{}+ \langle \beta_{\ep}(\varphi_{\ep}(t)), w_{\ep} \pcol{(t)}\rangle_{V^*, V} 
    + (\pi(\varphi_{\ep}(t)), w_{\ep}(t))_{H}.}
\end{align}
\pcol{Integration with respect to time of \eqref{dfPtauetaepsol2} gives 
\begin{equation}
\tau(\varphi_{\ep})_{t}(t) + \varphi_{\ep}(t) - \Delta_{N} w_{\ep}(t) 
= \tau v_{0\ep} + \varphi_{0\ep} \quad \hbox{for a.e. } t\in (0,T).
\label{pier12}
\end{equation}
If we test the above equation by $w_{\ep}$,
we obtain} 
\begin{align}\label{d2}
&\|\nabla w_{\ep}(t)\|_{H}^2 
\notag \\[2mm]
&= \pcol{{}- \tau\langle (\varphi_{\ep})_{t}(t), w_{\ep}(t)\rangle_{V^*, V} 
    - (\varphi_{\ep}(t), w_{\ep}(t))_{H} 
    + \langle \tau v_{0\ep} + \varphi_{0\ep}, w_{\ep}\pcol{(t)} \rangle_{V^*, V}}.
%\notag \\[2mm] 
%&= \sk{{}- \tau\langle (\varphi_{\ep})_{t}(t), 
%                    w_{\ep}(t) - m(w_{\ep}(t)) \rangle_{V^*, V}  
%    - \tau(m(v_{0})e^{-t/\tau}, w_{\ep}(t))_{H}}
%\notag \\
%&\,\quad\sk{{}- (\varphi_{\ep}(t), w_{\ep}(t))_{H} + \langle \tau v_{0\ep} + \varphi_{0\ep}, w_{\ep} \pcol{(t)}\rangle_{V^*, V}.}
\end{align}
Therefore, \pcol{we can show Lemma~\ref{estiep5} 
by adding \eqref{d1} and \eqref{d2}, then integrating with respect to time and using the Young inequality along with the estimates in Lemmas~\ref{estiep1}-\ref{estiep3}, finally applying the Gronwall lemma.}
\end{proof}

\begin{lem}\label{estiep6}
There exists a constant $C > 0$ such that 
\[
\pcol{\|w_{\ep}\|_{L^{\infty}(0, T; V)}} \leq C
\]
for all \sk{$\tau, \eta, \ep \in (0, 1]$.}
\end{lem}
\begin{proof}
\pcol{The assertion follows from \eqref{d2}, thanks to the Young inequality and Lem\-mas~\ref{estiep1} 
and~\ref{estiep5}}.
\end{proof}

\begin{lem}\label{estiep7}
\pcol{There exists a constant $C > 0$ such that } 
\begin{align*}
\pcol{\|(\theta_{\ep})_{t}\|_{L^2(0, T; V^*)} \leq C}    
\end{align*}
for all  \pcol{$\tau, \eta, \ep \in (0, 1]$.}
\end{lem}
\begin{proof}
\pcol{By comparing terms in~\eqref{dfPtauetaepsol1}, the above estimate results from Lemma~\ref{estiep1} 
and the boundedness of $\|f_{\ep}\|_{L^{2}(0, T; V^*)}$ from~\eqref{pier2}.}
\end{proof}

\begin{prth2.1}
\sk{Let $\tau, \eta \in (0, 1]$.} Then, \pcol{owing to} Lemmas~\ref{estiep1}-\ref{estiep7}, 
the compact embeddings $V \hookrightarrow H \hookrightarrow V^*$, 
the Aubin--Lions lemma \pcol{(see, e.g., \cite[Sect.~8,~Cor.~4]{Simon})}, 
there exist some functions 
%\sk{$\theta_{\eta}$, $w_{\eta}$, $\mu_{\eta}$, $\varphi_{\eta}$, $\xi_{\eta}$} 
%such that 
\begin{align*}
&\theta_{\eta} \in H^1(0, T; V^*) \cap L^{\infty}(0, T; H) \cap L^2(0, T; V), \\
&w_{\eta} \in \pcol{W^{1, \infty}(0, T; V^{*})} \cap L^{\infty}(0, T; V), \\
&\mu_{\eta} \in \pcol{L^{\infty}(0, T; V^*)}, \\
&\varphi_{\eta} \in W^{1, \infty}(0, T; V^{*}) \cap \sk{H^{1}(0, T; H)} 
                                                                \cap L^{\infty}(0, T; V), \\
&\xi_{\eta} \in L^{\infty}(0, T; L^q(\Omega)),
\end{align*}
\pcol{such that}
\begin{align}
&\theta_{\ep} \to \theta_{\eta} 
\quad \mbox{weakly* in}\ H^1(0, T; V^*) \cap L^{\infty}(0, T; H) \cap L^2(0, T; V), 
\label{wc1} \\
&\theta_{\ep} \to \theta_{\eta} 
\quad \mbox{strongly in}\ \pcol{C^0}([0, T]; \pcol{V^*}) \pcol{{}\cap L^2(0,T;H)} 
\label{st1} \\[2mm]
&\varphi_{\ep} \to \varphi_{\eta} 
\quad \mbox{weakly* in}\ W^{1, \infty}(0, T; V^{*}) \cap \sk{H^{1}(0, T; H)} 
                                                                          \cap L^{\infty}(0, T; V), 
\label{wc2} \\
&\varphi_{\ep} \to \varphi_{\eta} 
\quad \mbox{strongly in}\ \pcol{C^0}([0, T]; H), 
\label{st2} \\
&\beta_{\ep}(\varphi_{\ep}) \to \xi_{\eta} 
\quad \mbox{weakly* in}\ L^{\infty}(0, T; L^q(\Omega)), 
\label{wc3} \\[2mm]
&w_{\ep} \to w_{\eta} 
\quad \mbox{weakly* in}\ L^{\infty}(0, T; V), 
\label{wc4} \\
&(w_{\ep})_{t}\ (= \mu_{\ep}) \to \mu_{\eta} \ \pcol{(=(w_{\eta})_{t})}
\quad \pcol{\mbox{weakly in}\ L^{\infty}(0, T; V^*)}, 
\label{wc5} \\
&w_{\ep} \to w_{\eta} 
\quad \mbox{strongly in}\ \pcol{C^0}([0, T]; H)
\label{st3}
\end{align}
\pcol{along a subsequence of~$\ep = \ep_{j} \searrow 0$. Now, we have to show that 
$(\theta_{\eta}, \mu_{\eta},\varphi_{\eta})$ is a weak solution of the problem~\ref{Ptaueta}.}

\pcol{To this aim, we note that from \eqref{st2}, \sk{the property of $J_{\ep}^{\beta}$}, 
the continuity of $\beta$ it follows that} 
\sk{\[
  \beta_{\ep_j}(\varphi_{\ep_j})\ (= \beta(J_{\ep_j}^{\beta}(\varphi_{\ep_j})))\to \beta(\varphi_{\eta}) \quad \mbox{a.e.\ in}\ \Omega\times(0, T)
\]}%
as $\ep_{j} = \ep_{j_{k}} \searrow 0$.  
Then the Severini--Egorov theorem and Lemma \ref{estiep3} 
yield that 
\sk{\begin{equation}\label{wc6}
  \beta_{\ep_j}(\varphi_{\ep_j})\to \beta(\varphi_{\eta}) \quad \text{weakly in } L^q(\Omega\times(0,T))
\end{equation}}%
and 
\sk{\begin{equation}\label{st4}
  \beta_{\ep_j}(\varphi_{\ep_j})\to \beta(\varphi_{\eta}) \quad \text{strongly in } L^p(\Omega\times(0,T))
  \quad\forall\,p\in[1,q)
\end{equation}}%
as $\ep = \ep_{j_{k}} \searrow 0$ 
(the reader may see \cite[Lemme~1.3, p.12]{Lio69} for some detail). 
Thus we deduce from \eqref{wc3} and \eqref{wc6} that 
\sk{\begin{equation}\label{xibeta}
\xi_{\eta} = \beta(\varphi_{\eta})  \quad \mbox{a.e.\ in}\ \Omega\times(0, T).
\end{equation}}%
\pcol{At this point, we can multiply \eqref{dfPtauetaepsol1} by an arbitrary $v\in L^2 (0,T;V)$
and integrate over $\Omega \times (0, T) $ to deduce that 
\begin{align*}
&\int_{0}^{T} \langle (\theta_{\eta})_{t}, v (t) \rangle_{V^{*}, V}dt  
  + \int_{0}^{T} \!\! \int_{\Omega}\nabla\theta_{\eta}\cdot\nabla v 
= \int_{0}^{T}\langle (f_\ep  - (\varphi_{\ep})_{t}(t), v(t) \rangle_{V^{*}, V}dt  
     %  \qquad \mbox{a.e.\ on}\ (0, T) \quad \mbox{for all}\ v \in V, 
\end{align*}   
whence \eqref{dfPtauetaepsol1} follows as $\ep = \ep_{j} \searrow 0$, thanks to \eqref{wc1}, \eqref{wc2} and \eqref{pier2}. Note that the formulation with test function $v\in L^2 (0,T;V)$
and the formulation a.e. in $(0,T)$ for all $v\in  V$ are equivalent.}

\pcol{A similar procedure applied to \eqref{pier12}, which is the time-integrated version of \eqref{dfPtauetaepsol2}, leads to \eqref{dfPtauetasol2}: in this step we also use the convergence~\eqref{pier3} of the initial values. Besides, in view of \eqref{wc3} and \eqref{xibeta} 
no trouble arises in the passage 
to the limit in \eqref{dfPtauetaepsol3}, so that we arrive at \eqref{dfPtauetasol3}. On the other hand, 
the initial conditions~\eqref{dfPtauetasol4} are a plain consequence of 
\eqref{dfPtauetaepsol4}, \eqref{st1}, \eqref{st2} and \eqref{pier3}.}

\pcol{Moreover, we point out that the regularity \eqref{etasp2}, i.e., $\mu_{\eta}= (w_{\eta})_t \in \pcol{L^{\infty}(0, T; V^*)}$, follows from a comparison of all terms in the variational equality~\eqref{dfPtauetaepsol3}, as in particular $\eta (\varphi_{\eta})_t \in L^\infty (0, T; V^*)$ by \eqref{wc2}.}

\pcol{Therefore, Theorem~\ref{maintheorem1} is completely proved.} 
\qed
\end{prth2.1}

%%==============================================================%%
%%==============                                  ==============%%
%%======                      Section7                    ======%%
%%====                                                      ====%%
%%==                                                          ==%%
%%====                                                      ====%%
%%======                                                  ======%%
%%==============                                  ==============%%
%%==============================================================%%

\section{\pcol{Asymptotic results}}
\label{Sec7}

This section is devoted to the asymptotic convergences as the two main parameters $\tau, \eta$ tend to zero. 

\subsection{\pcol{Passage to the limit as $\eta\searrow0$}}
\label{Sec7.1}

\sk{\pcol{First, we recall the a priori estimates} for \ref{Ptaueta} 
and prove Theorem~\ref{maintheorem1} 
by passing to the limit in \ref{Ptaueta} as $\eta \searrow 0$.}
\sk{\begin{lem}\label{estieta1}
There exists a constant $C > 0$ such that 
\begin{align*}
&\pcol{\|\theta_{\eta}\|_{L^{\infty}(0, T; H)\cap L^2(0,T;V) } }
% + \|\nabla\theta_{\eta}\|_{L^2(0, T; H)} 
+ \eta^{1/2}\|(\varphi_{\eta})_{t}\|_{L^2(0, T; H)} 
+ \tau^{1/2}\|(\varphi_{\eta})_{t}\|_{L^{\infty}(0, T; V^{*})} 
\\[1mm]
&+ \pcol{\|(\varphi_{\eta})_{t}\|_{L^2(0, T; V^{*})}} + \|\varphi_{\eta}\|_{L^{\infty}(0, T; V)} 
+ \|\beta(\varphi_{\eta})\|_{L^{\infty}(0, T; L^q(\Omega))} 
\leq C
\end{align*}
for all $\tau, \eta \in (0, 1]$.
\end{lem}}%
\sk{\begin{proof}
\pcol{This lemma holds true thanks to Lemmas~\ref{estiep1}-\ref{estiep3} and to the weak or weak* lower semicontinuity of norms in the passage to the limit of the previous section.}
\end{proof}}%
\sk{\begin{lem}\label{estieta2}
\pcol{There exists a constant $C > 0$ such that}
\begin{align}\label{estieta2-1}
\pcol{\|\mu_{\eta}\|_{L^2(0, T; V^*)} 
+ \|w_{\eta}\|_{L^{\infty}(0, T; V)} +{}}
\|(\theta_{\eta})_{t}\|_{L^2(0, T; V^*)} \leq C   
\end{align}
\pcol{for all $\tau, \eta \in (0, 1]$.}
Also, for all $\tau \in (0, 1]$ there exists a constant $C_{\tau} > 0$ such that 
\begin{align}\label{estieta2-2}
\|\mu_{\eta}\|_{L^{\infty}(0, T; V^*)} \leq C_{\tau}
\end{align}
for all $\eta \in (0, 1]$.
\end{lem}}%
\sk{\begin{proof}
\pcol{It suffices to recall Lemmas~\ref{estiep4}-\ref{estiep7} and, still, the weak or weak* lower semicontinuity of norms.}
\end{proof}}%
\sk{\begin{prth2.2}
We can prove Theorem~\ref{maintheorem2} \pcol{arguing 
as in the proof of Theorem~\ref{maintheorem1}, by passing to the limit for a subsequence of $\eta = \eta_j$ tending to $0$ and finding a weak solution $(\theta_{\tau}, \mu_{\tau},\varphi_{\tau})$ to the problem~\ref{Ptau}. The only difference is that now the convergence in \eqref{wc2} is replaced by 
\begin{align*}
&\varphi_{\eta} \to \varphi_{\tau} 
\quad \mbox{weakly* in}\ W^{1, \infty}(0, T; V^{*}) \cap L^{\infty}(0, T; V) \\
&\hbox{and} \quad \eta \varphi_{\eta} \to 0
\quad \mbox{strongly in}\ \pcol{H^{1}(0, T; H)}, 
\end{align*}
but the strong convergence $\varphi_{\eta} \to \varphi_{\tau} $
in $\pcol{C^0}([0, T]; H)$ remains true also in this case. Everything else works as before.} 
\qed 
\end{prth2.2}}
%%==============================================================%%
%%==============                                  ==============%%
%%======                      Section8                    ======%%
%%====                                                      ====%%
%%==                                                          ==%%
%%====                                                      ====%%
%%======                                                  ======%%
%%==============                                  ==============%%
%%==============================================================%%

\subsection{Convergence as $\tau \searrow 0$} 
\label{Sec7.2}

\pcol{Now, we fix $\eta$ and either we let $\eta \in (0, 1]$ or $\eta= 0$. Hence,   
we also introduce the problem
\begin{equation}\label{Pzeroeta}\tag*{(P)$_{0\eta}$}
\begin{cases}
\theta_t + \varphi_t - \Delta\theta = f  & \mbox{in}\ \Omega\times(0, T), \\[1mm]
\varphi_{t} - \Delta\mu = 0 
& \mbox{in}\ \Omega\times(0, T), \\[1mm]
\eta\varphi_{t} - \Delta\varphi + \beta(\varphi) + \pi(\varphi) = \mu + \theta 
& \mbox{in}\ \Omega\times(0, T), \\[1mm]
\partial_{\nu} \theta =  \partial_{\nu} \mu = \partial_{\nu} \varphi = 0
& \mbox{on}\ \partial\Omega\times(0, T), \\[1mm]
\theta(0) = \theta_0,\ \varphi(0) = \varphi_0 
& \mbox{in}\ \Omega,  
\end{cases}
\end{equation}
and define weak solutions as follows.} 
\sk{\begin{df}
A triplet $(\theta, \mu, \varphi)$ with 
\begin{align}
&\theta \in H^1(0, T; V^*) \cap L^{\infty}(0, T; H) \cap L^2(0, T; V), \label{pier13}
\\
&\displaystyle t \mapsto  w(t):=\int_{0}^{t} \mu(s)\,ds \ \hbox{ is in }\ 
\pcol{H^1(0, T; V)},%\cap L^{\infty}(0, T; V),
\label{pier14}
\\
&\pcol{\varphi \in H^{1}(0, T; V^*) \cap L^{\infty}(0, T; V), \qquad \eta \hskip1pt\varphi \in H^{1}(0, T; H),}  \label{pier15}
\\[1mm]
&\beta(\varphi) \in L^{\infty}(0, T; L^q(\Omega))  \label{pier16}
\end{align}
is called a {\it weak solution} of \ref{Pzeroeta} if 
\begin{align}
&\langle \theta_{t}, v \rangle_{V^{*}, V} 
  + \int_{\Omega}\nabla\theta\cdot\nabla v 
\notag \\ 
  &= \langle f - \varphi_{t}, v \rangle_{V^{*}, V} 
       \qquad \mbox{a.e.\ on}\ (0, T) \quad \mbox{for all}\ v \in V,  \label{pier17}
\\[4mm] 
&\pcol{\langle \varphi_t, v \rangle_{V^{*}, V} 
   + \int_{\Omega}\nabla w_t \cdot \nabla v = 0}
     \qquad \mbox{a.e.\ on}\ (0, T) \quad \mbox{for all}\ v \in V,  \label{pier18}
\\[4mm] 
&\eta(\varphi_{t}, v)_{H} 
+ \int_{\Omega} \nabla \varphi \cdot \nabla v 
  + \langle \beta(\varphi), v \rangle_{V^{*}, V}  
\notag \\ 
&= \langle w_{t}, v \rangle_{V^{*}, V} 
     + (\theta - \pi(\varphi), v)_{H}
         \qquad \mbox{a.e.\ on}\ (0, T) \quad \mbox{for all}\ v \in V,   \label{pier19}
\\[4mm]
&\theta(0) = \theta_{0},\ \varphi(0) = \varphi_{0}  \label{pier20}
        \quad \mbox{a.e.\ in}\ \Omega.
\end{align}
\end{df}}%

\bigskip

\pcol{Next, we let $(\theta_{\tau}, \mu_{\tau}, \varphi_{\tau})$ 
be either the weak solution of the problem \ref{Ptaueta} found by Theorem~\ref{maintheorem1},
if $\eta \in (0, 1]$, or the weak solution of the problem \ref{Ptau} found by Theorem~\ref{maintheorem2},
if $\eta=0$.
Then, we can state the following properties.  
\begin{lem}\label{estitau1}
There exists a constant $C > 0$ such that 
\begin{align*}
&\pcol{\|\theta_{\tau}\|_{H^1 (0,T;V^*) \cap L^{\infty}(0, T; H)\cap L^2(0,T;V) } }
% + \|\nabla\theta_{\tau}\|_{L^2(0, T; H)} 
+ \eta^{1/2}\|(\varphi_{\tau})_{t}\|_{L^2(0, T; H)} 
\\[1mm]
&+ \tau^{1/2}\|(\varphi_{\tau})_{t}\|_{L^{\infty}(0, T; V^{*})} 
+ \pcol{\|\varphi_{\tau}\|_{H^1(0, T; V^{*})\cap L^{\infty}(0, T; V)}}
\\[1mm]
&+ \|\beta(\varphi_{\tau})\|_{L^{\infty}(0, T; L^q(\Omega))} 
+\|\mu_{\tau}\|_{L^2(0, T; V^*)} 
+ \|w_{\tau}\|_{L^{\infty}(0, T; V)} \leq C   
\end{align*}
\pcol{for all $\tau \in (0, 1]$.}
\end{lem}
\begin{proof}
\pcol{These uniform estimates follow from Lemmas~\ref{estiep1}-\ref{estiep7} if $\eta $ is positive, and  from Lemmas~\ref{estieta1}-\ref{estieta2} if $\eta=0$, of course taking into account the weak or weak star lower semicontinuity of norms.}
\end{proof}}%

\bigskip

\sk{\noindent Therefore we can prove the following asymptotic result. 
\begin{thm}\label{maintheorem4}
Assume that {\rm (C1)-(C3)} hold 
and \pcol{let the triple {\rm $(\theta_{\tau}, \mu_{\tau}, \varphi_{\tau})$} 
be as above for $\tau \in (0, 1]$.  
Then, either for  $\eta \in (0, 1]$ or for $\eta =0$,}
there exists a weak solution $(\theta, \mu, \varphi)$ 
of the problem {\rm \ref{Pzeroeta}} 
and a sequence {\rm $\{\tau_{j}\}_{j}$} such that, as $j \to \infty$, $\tau_{j} \searrow 0$ and 
\begin{align*}
\theta_{\tau_{j}} \to \theta 
\quad &\pcol{\mbox{weakly* in}\ H^1(0, T; V^*) \cap L^{\infty}(0, T; H) \cap L^2(0, T; V),}
\\
\theta_{\tau_{j}} \to \theta 
\quad &\pcol{\mbox{strongly in}\ {C^0}([0, T]; \pcol{V^*}) \pcol{{}\cap L^2(0,T;H)} } 
\\[2mm]
\tau_{j}(\varphi_{\tau_{j}})_{t} \to 0 
\quad &\mbox{strongly in}\ L^{\infty}(0, T; V^{*}), 
\\
\varphi_{\tau_{j}} \to \varphi 
\quad &\mbox{weakly* in}\ \pcol{H^{1}(0, T; V^*)} \cap L^{\infty}(0, T; V), 
\\
\pcol{\varphi_{\tau_{j}} \to \varphi}
\quad &\pcol{\mbox{strongly in}\ \pcol{C^0}([0, T]; H),}
\\
\pcol{\eta\hskip1pt \varphi_{\tau_{j}} \to \eta\hskip1pt \varphi }
\quad &\pcol{\mbox{weakly in}\ H^{1}(0, T; H),} 
\\
\beta(\varphi_{\tau_{j}}) \to \beta(\varphi) 
\quad &\mbox{weakly in}\ L^{\infty}(0, T; L^{q}(\Omega)), 
\\[2mm]
\mu_{\tau_{j}} \to \mu 
\quad &\pcol{\mbox{weakly in}\ L^2(0, T; V^{*}),} 
\\
w_{\tau_{j}} \to w 
\quad &\pcol{\mbox{weakly* in}\ L^{\infty}(0, T; V) \ \mbox{and strongly in}\ \pcol{C^0}([0, T]; H).}
\end{align*}
\end{thm}}%
\begin{proof}
\pcol{By the uniform estimates in Lemma~\ref{estitau1}, we can argue by compactness and infer that the 
above convergences hold. Next, we can check that $(\theta, \mu, \varphi)$ satisfies \eqref{pier17}, 
\eqref{pier19}, \eqref{pier20} and the time-integrated version of \eqref{pier18}: for this we proceed 
similarly as in the proof of Theorem~\ref{maintheorem1}.
At this point, in order to conclude that the limit $(\theta, \mu, \varphi)$ yields a weak solution to the 
problem \ref{Pzeroeta}, it is enough to check that \eqref{pier14} holds, that is, 
$\mu = w_t $ is in $L^2(0,T;V)$. However, being already $\mu\in L^2(0,T;V^*)$, then with the mean value 
under control, the regularity $\mu \in L^2(0,T;V)$ follows from a comparison in \eqref{pier18}  and the fact that $\varphi_t \in  L^2(0,T;V^*).$}
\end{proof}

\section*{Acknowledgments}
\pier{\pcol{The authors are very grateful to Professor Takeshi Fukao for his valuable 
comments and suggestions on a previous version of this paper.}
The research of PC has been performed within the framework 
of the MIUR-PRIN Grant 2020F3NCPX ``Mathematics for industry 4.0 (Math4I4)''.
In addition, PC indicates his affiliation 
to the GNAMPA (Gruppo Nazionale per l'Analisi Matematica, 
la Probabilit\`a e le loro Applicazioni) of INdAM (Isti\-tuto 
Nazionale di Alta Matematica).} 
The research of SK is supported 
by JSPS KAKENHI Grant Number JP23K12990.

%
%
%%==============================================================%%
%%==============                                  ==============%%
%%======                                                  ======%%
%%====                                                      ====%%
%%==                         Reference                        ==%%
%%====                                                      ====%%
%%======                                                  ======%%
%%==============                                  ==============%%
%%==============================================================%%

\end{document}